%% file: arxiv_ICML.tex
\documentclass[uplatex,english, a4paper]{article}  

\usepackage[square,comma, sort&compress]{natbib}
\usepackage[margin=1in]{geometry}
\usepackage{jmlr2e}

\usepackage{amsfonts}       % blackboard math symbols
\usepackage{microtype}      % microtypography

\usepackage{charter}

\input{preamble.tex}

\usepackage[vvarbb]{newtxmath}
\usepackage[capitalize,nameinlink]{cleveref}[0.21]

\title{Sparser Kernel Herding with Pairwise Conditional Gradients \\ without Swap Steps}

\date{}

\author{\name Kazuma Tsuji 
	\email  \href{mailto:kazuma_tsuji@mist.i.u-tokyo.ac.jp}{kazuma\_tsuji@mist.i.u-tokyo.ac.jp}\\
	\addr Graduate School of Information Science and Technology\\ 
	The University of Tokyo \\ 
	Tokyo, Japan
	\AND
	\name Ken'ichiro Tanaka \email \href{mailto:kenichiro@mist.i.u-tokyo.ac.jp}{kenichiro@mist.i.u-tokyo.ac.jp} \\
	\addr Graduate School of Information Science and Technology\\ 
	The University of Tokyo \\
	PRESTO \\
	Japan Science and Technological Agency (JST) \\ 
	Tokyo, Japan
	\AND
	\name Sebastian Pokutta \email \href{mailto:pokutta@zib.de}{pokutta@zib.de} \\
	\addr Institute of Mathematics\\
	Zuse Institute Berlin and Technische Universität Berlin\\
	Berlin, Germany}

\usepackage{nameref} 
\usepackage{zref-xr}
\zxrsetup{toltxlabel} 
\zexternaldocument*[1:]{supplement} 

\begin{document}
\maketitle

\begin{abstract}
The Pairwise Conditional Gradients (PCG) algorithm is a powerful extension of the Frank-Wolfe algorithm leading to particularly sparse solutions, which makes PCG very appealing for problems such as sparse signal recovery, sparse regression, and kernel herding. Unfortunately, PCG exhibits so-called swap steps that might not provide sufficient primal progress. The number of these bad steps is bounded by a function in the dimension and as such known guarantees do not generalize to the infinite-dimensional case, which would be needed for kernel herding. We propose a new variant of PCG, the so-called Blended Pairwise Conditional Gradients (BPCG). This new algorithm does not exhibit any swap steps, is very easy to implement, and does not require any internal gradient alignment procedures. The convergence rate of BPCG is basically that of PCG if no drop steps would occur and as such is no worse than PCG but improves and provides new rates in many cases. Moreover, we observe in the numerical experiments that BPCG’s solutions are much sparser than those of PCG. We apply BPCG to the kernel herding setting, where we derive nice quadrature rules and provide numerical results demonstrating the performance of our method.
\end{abstract}

\section{Introduction}

Conditional Gradients \citep{polyak66cg} (also: Frank-Wolfe algorithms \citep{fw56}) are an important class of first-order methods for constrained convex minimization, i.e., solving 
$$\min_{x \in P} f(x),$$ 
where $P$ is a compact convex feasible region. These methods usually form their iterates as convex combinations of feasible points and as such do not require (potentially expensive) projections onto the feasible region $P$. Moreover the access to the feasible region is solely realized by means of a so-called \emph{Linear Minimization Oracle (LMO)}, which upon presentation with a linear function $c$ returns $\arg\min_{x \in P} c^\intercal x$. Another significant advantage is that the iterates are typically formed as \emph{sparse} convex combinations of extremal points of the feasible region (sometimes also called \emph{atoms}) which makes this class of optimization algorithms particularly appealing for problems such as sparse signal recovery, structured regression, SVM training, and also kernel herding. Over the recent years there have been significant advances in Frank-Wolfe algorithms providing even faster convergence rates and higher sparsity (in terms of the number of atoms participating in the convex combination) and in particular the \emph{Pairwise Conditional Gradients (PCG)} algorithm \citep{lacoste15} provides a very high convergence speed (both theoretically and in computations) and sparsity. The PCG algorithm exhibits so-called \emph{swap steps}, which are steps in which weight is shifting from one atom to another. These steps do not guarantee sufficient primal progress and hence usually convergence analyses bound their number in order to argue that there is a sufficient number of steps with good progress. Unfortunately, these bounds depend exponentially on the dimension of the feasible region \citep{lacoste15} and require $P$ to be a polytope for this bound to hold at all. This precludes application of PCG to the infinite-dimensional setting and even in the polyhedral setting the worst-case bound might be unappealing. Recently, several works \citep{rinaldi2020unifyfree,combettes20boostfw,MGP2020walking} suggested \lq{}enhancement procedures\rq{} to potentially overcome swap steps by improving the descent directions, however at the cost of significantly more complex algorithms and (still often dimension-dependent) analysis. In contrast, we propose a much simpler modification of the PCG algorithm by combining it with the blending criterion from the Blended Conditional Gradients Algorithm (BCG) of \citet{pok18bcg}. This modified PCG algorithm, which we refer to as \emph{Blended Pairwise Conditional Gradients} does not exhibit swap steps anymore and the convergence rates that we obtain are that which the original PCG algorithm would achieve if swap steps would not occur. As such it improves the convergence rates of PCG and moreover, by eschewing swap steps, this modification provides natural convergence rates for the infinite-dimensional setting, which is important for our application to kernel herding \citep{welling2009herding, 10.5555/3023549.3023562, 10.5555/3042573.3042747}. \emph{Kernel herding} is a particular method for constructing a quadrature formula by running a Conditional Gradient algorithm on the convex subset of a Reproducing Kernel Hillbert Space (RKHS); the obtained solution and the associated convex combination correspond to the respective approximation. It can be considered within the framework of kernel quadrature methods that have been studied in a long line of works, such as e.g., \cite{10.5555/3020652.3020694, oettershagen2017construction, 10.1214/18-STS683}. Accurate quadrature rules with a \emph{small} number of nodes are often desired, and as such the achievable sparsity of a given Conditional Gradient method when used for kernel herding is crucial.

\subsection*{Related Work}

There is a broad literature on conditional gradient algorithms and recently this class of methods regained significant attention with many new results and algorithms for specific setups. Most closely related to our work however are the Pairwise Conditional Gradients algorithm introduced in \citep{lacoste15} as well as the Away-step Frank-Wolfe algorithm \citep{wolfe70,gm86}; the Pairwise Conditional Gradients algorithm arises as a natural generalization of the Away-step Frank-Wolfe algorithm. Moreover, we used the blending criterion from \cite{pok18bcg} to efficiently blend together local PCG steps with global Frank-Wolfe steps.

On the other hand there is only a limited number of works attempting to use and extend Conditional Gradients to kernel herding. \citet{lacoste2015sequential} studied the practical performance of several variants of kernel herding that correspond to the variants of Conditional gradients, such as the Away-step Frank-Wolfe algorithm and Pairwise Conditional Gradients. More recently, \citet{tsuji2021acceleration} proposed a new variant of kernel herding with the explicit goal of obtaining sparse solutions. 

\subsection*{Contribution}

Our contribution can be roughly summarized as follows.

\emph{Pairwise Conditional Gradients without swap steps.} We present a new Pairwise Conditional Gradients algorithm, the Blended Pairwise Conditional Gradients (BPCG) (Algorithm~\ref{alg:BPCG}), that does not exhibit swap steps and provide convergence analyses of this method. In particular, the BPCG algorithm provides improved rates compared to PCG; the lack of swap steps makes the difference of the constant factor. Here we focus on the general smooth convex case and the strongly convex case over polytopes. The results for the general smooth convex case is also applicable to the infinite-dimensional general convex domains.  We hasten to stress though that our convergence analyses can be immediately extended to sharp functions (a generalization of strongly convex functions) as well as uniformly convex feasible regions by combining \citep{kerdreux2018restarting} and \citep{UniformConvexFW_2020}, respectively, with our arguments here; due to space limitations this is beyond the scope of this paper. Additionally, numerical experiments suggest that the BPCG algorithm outputs fairly sparse solutions practically. BPCG offers superior sparsity of the iterates both due to the pairwise steps as well as the BCG criterion that favors local steps, which maintain the sparsity of the solution. 

We also provide a lazified version (see \cite{pok17lazy,BPZ2017jour}) of BPCG (Algorithm~\ref{alg:lazifiedBPCG}) that significantly reduces the number of LMO calls by reusing previously visited atoms at the expense of a small constant factor loss in the convergence rates. The lazified variant is in particular useful when the LMO is expensive as the number of required calls to the LMO can be reduced dramatically in actual computations; see e.g., \citep{pok17lazy,BPZ2017jour,pok18bcg} for the benefits of lazification.

\emph{Sparser Kernel Herding.}
We demonstrate the effectiveness of applying our BPCG algorithm to kernel herding. From a theoretical perspective, we can apply the convergence guarantees for the general smooth convex case mentioned before and obtain state-of-the-art guarantees. Moreover, in numerical experiments we demonstrate that the practical performance of BPCG for kernel herding much exceed the theoretical guarantees. BPCG and the lazified version achieve very fast convergence in the number of nodes which are competitive with optimal convergence rates. In addition, the lazification contributes to reducing computational cost because the LMO in kernel herding is computationally rather expensive.

\emph{Computational Results.} We complement our theoretical analyses as well as the numerical experiments for kernel herding with general purpose computational experiments demonstrating the excellent performance of BPCG across several problems of interest.

\section{Preliminaries}

Let $f: \mathbf{R}^{d} \to \mathbf{R}$ be a differentiable, convex, and $L$-smooth function. 
Recall that $f$ is \emph{$L$-smooth} if 
\begin{align}
f(y) - f(x) - \left< \nabla f(x), \, y-x \right> \leq \frac{L}{2} \| y - x \|^{2}
\notag
\end{align}
for all $x, y \in \mathbf{R}^{d}$. 
In addition, 
$f$ is \emph{$\mu$-strongly convex} if 
\begin{align}
f(y) - f(x) - \left< \nabla f(x), \, y-x \right> \geq \frac{\mu}{2} \| y - x \|^{2}
\notag
\end{align}
for all $x, y \in \mathbf{R}^{d}$. 
Let $P \subset \zissu^d$ be a convex feasible region and the subset  $V(P) \subset P$ satisfy $P= \mathrm{conv}(V(P))$. The notation $\mathrm{conv}(V(P))$ means the convex hull of $V(P)$. 
We assume that $P$ is bounded and its diameter $D$ is given by $D := \sup_{x, y \in P} \| x - y \|$. 
Furthermore, 
let $\delta_P > 0$ be the pyramidal width of $P$ \citep{lacoste15}; we drop the index if the feasible region is clear from the context. For a convex combination $x = \sum_{i=1}^{n} c_{i} \, v_{i}$, as later maintained by the algorithm, let $c[x](v_{i}) \coloneqq c_{i}$. In the following let $x^{\ast} \in P$ denote the (not necessarily unique) minimizer of $f$ over $P$.

%------

\section{Blended Pairwise Conditional Gradients Algorithm}

We will now present the algorithm and its convergence analysis.

\subsection{Algorithm}

We consider the Blended Pairwise Conditional Gradients algorithm (BPCG) shown in Algorithm~\ref{alg:BPCG}. The BPCG algorithm is the same type of algorithm as the Blended Conditional Gradients algorithm in \cite{pok18bcg}. In Algorithm~\ref{alg:BPCG}, if the local pairwise gap $\left<\nabla f(x_t), a_t - s_t \right>$ is smaller than the Frank-Wolfe gap $\left<\nabla f(x_t), x_t - w_t \right>$, a \emph{FW step} (line \ref{FW_step_start}-\ref{alg:updateSt_FW}) is taken. Otherwise, the weights of the active atoms in $S_t$ are optimized by the Pairwise Conditional Gradients (PCG) \emph{locally}. If the step size $\lambda_t$ is larger than $\Lambda_t ^\ast$, the away vertex is removed from the active set $S_t$ and we cal the step \emph{drop step} (line \ref{drop step}). Otherwise we call the step \emph{descent step} (line \ref{descent step}). Descent step and drop step are referred to as \emph{pairwise step} all together.

By the structure of the BPCG algorithm, the sparsity of the solutions is expected since the new atoms are not added to $S_t$ until the local pairwise gap decreases sufficiently. Moreover, since the PCG is implemented locally, BPCG does not exhibit \emph{swap steps} in which weight is shifting from the away atom to the Frank-Wolfe atom. In the PCG, swap steps do not show much progress and the number of swap steps is bounded by the dimension-dependent constant. Therefore, the local implementation of the PCG is significant to extend the pairwise type algorithms to infinite-dimensional cases.

Note that both Algorithm~\ref{alg:BPCG} and Algorithm~\ref{alg:lazifiedBPCG} which is introduced in \autoref{subsection_lazified}  use line search here to simplify the presentation. However both can be run alternatively with the short-step rule, which minimizes the quadratic inequality arising from smoothness (this is precisely the $\lambda^*_t$ in Lemma~\ref{thm:not_drop_step}) but requires knowledge of $L$ or with the adaptive step-size strategy of \citep{pedregosa2018step}, which offers a performance similar to line search at the fraction of the cost; our analysis applies to these two step-size strategies as well.

%\begin{algorithm}
\begin{algorithm}[ht]
\caption{Blended Pairwise Conditional Gradients (BPCG)}
\label{alg:BPCG}
\begin{algorithmic}[1]
\REQUIRE convex smooth function $f$, start vertex $x_0 \in V(P)$.  
\ENSURE points $x_{1}, \ldots, x_T$ in $P$. 
\STATE $S_{0} \leftarrow \{ x_0 \}$
\FOR {$t = 0$ to $T-1$}
	\STATE $a_t \leftarrow \argmax_{v \in S_t} \left< \nabla f(x_t), v \right>$ \quad \COMMENT {away vertex}
	\STATE $s_t \leftarrow \argmin_{v \in S_t} \left< \nabla f(x_t), v \right>$ \quad \COMMENT {local FW}
	\STATE $w_t \leftarrow \argmin_{v \in V(P)} \left< \nabla f(x_t), v \right>$ \quad \COMMENT {global FW}
	\IF {$\left< \nabla f(x_t), a_t -s_t \right> \geq \left< \nabla f(x_t), x_t -w_t \right>$ } \label{eq:select}
		\STATE $d_t = a_t - s_t$
		\STATE $\Lambda_{t}^{\ast} \leftarrow c[x_{t}](a_{t})$
		\STATE $\lambda_t \leftarrow \argmin_{\lambda \in [0, \, \Lambda_{t}^{\ast}]} f(x_{t} - \lambda d_{t})$ \label{alg:PW}
		\IF {$\lambda_{t} < \Lambda_{t}^{\ast}$}
			\STATE $S_{t+1} \leftarrow S_{t}$  \quad \COMMENT {descent step}\label{descent step}
		\ELSE
			\STATE $S_{t+1} \leftarrow S_{t} \setminus \{ a_{t} \}$ \quad \COMMENT {drop step} \label{drop step}
		\ENDIF
	\ELSE 
		\STATE $d_t = x_t - w_t$ \label{FW_step_start}
		\STATE $\lambda_t  \leftarrow \argmin_{\lambda \in [0, \, 1]} f(x_t - \lambda d_t)$  \label{alg:FW}
		\STATE $S_{t+1} \leftarrow S_t \cup \{ w_t \}$ (or $S_{t+1} \leftarrow \{w_t\}$ if $\lambda_t = 1$) \quad \COMMENT {FW step} \label{alg:updateSt_FW} 
	\ENDIF
	\STATE $x_{t+1} \leftarrow x_t - \lambda_t d_t$
\ENDFOR
\end{algorithmic}
\end{algorithm}

\subsection{Convergence analysis}

The following theorems provide convergence properties of the BPCG algorithm. We first state the general smooth case.

\begin{thm}
\label{thm:only_Lsmooth_Tinv_convergence} Let $P$ be a convex feasible domain of diameter $D$. Assume that $f$ is convex and $L$-smooth. 
Let $\{ x_{i} \}_{i=0}^{T} \subset P$ be the sequence given by the BPCG algorithm (Algorithm~\ref{alg:BPCG}). 
Then, it holds that
\begin{align}
f(x_{T}) - f(x^{\ast}) \leq \frac{4LD^{2}}{T}.
\end{align}
\end{thm}

In the case of strongly convex functions $f$ and polyhedral feasible regions $P$ we obtain the following improved convergence rates. Note that in contrast to the analysis of the Pairwise Conditional Gradients algorithm in \citep{lacoste15} we do not encounter swap steps. 

\begin{thm}
	\label{thm:mu_strong_linear_convergence} Let $P$ be a polytope with pyramidal width $\delta$ and diameter $D$. Furthermore, let $f$ be $\mu$-strongly convex and $L$-smooth and consider the sequence $\{ x_{i} \}_{i=0}^{T} \subset P$ obtained by the BPCG algorithm (Algorithm~\ref{alg:BPCG}).
	Then, it holds that
	\begin{align}
		f(x_{T}) - f(x^{\ast}) \leq (f(x_{0}) - f(x^{\ast})) \, \exp \left( - c_{f, P} \, T \right), 
	\end{align}
	where $c_{f,P} := \frac{1}{2} \min \{ \frac{1}{2}, \, \frac{\mu \delta^{2}}{4LD^{2}} \}$.
\end{thm}

We prove these theorems by using the following lemmas. 
%In the %following we refer to a pairwise step with %$\lambda_{t}=\Lambda_{t}^{\ast}$ 
%as a \emph{drop step} because the vertex $a_{t}$ is dropped (i.e., %removed) from the active set $S_t$. If $\lambda_{t} < %\Lambda_{t}^{\ast}$, we refer to the pairwise step as a %\emph{descent step}. In addition, we call the step in which the %Frank-Wolfe direction $d_t = x_t - w_t$ is used \emph{FW step}. 

% Preliminary lemma
\begin{lem}[Geometric Strong Convexity, (\cite{lacoste15}, Inequalities (23) and (28) )]
\label{thm:mu_strong_pyramidal_width}
Assume that $f$ is $\mu$-strongly convex and $P$ is a polytope with pyramidal width $\delta$. 
Then, with the notation of Algorithm~\ref{alg:BPCG} the following inequality holds:
\begin{align}
f(x_{t}) - f(x^{\ast}) \leq \frac{\left< \nabla f(x_t), a_t - w_t \right>^{2}}{2 \mu \delta^{2}}.
\end{align}
\end{lem}

% First lemma
\begin{lem}
\label{thm:not_drop_step}
With the notation of Algorithm~\ref{alg:BPCG}, suppose that step $t$ is not a drop step (line \ref{drop step}). 
Let 
$\lambda_{t}^{\ast} = \frac{\langle \nabla f(x_{t}), \, d_{t} \rangle}{L\| d_{t} \|^{2}}$. 
\begin{enumerate}
\item[(a)]
If step $t$ is either 
a FW step (line \ref{FW_step_start}-\ref{alg:updateSt_FW}) with $\lambda_{t}^{\ast} < 1$ or
a descent step (line \ref{descent step}),
we have
\begin{align}
f(x_{t}) - f(x_{t+1}) \geq \frac{\langle \nabla f(x_{t}), \, d_{t} \rangle^{2}}{2LD^{2}}.
\label{eq:PW_FW_with_small_lambda}
\end{align}

\item[(b)]
If step $t$ is an FW step (line \ref{FW_step_start}-\ref{alg:updateSt_FW}) with $\lambda_{t}^{\ast} \geq 1$, 
we have
\begin{align}
f(x_{t}) - f(x_{t+1}) \geq \frac{1}{2} \langle \nabla f(x_{t}), \, d_{t} \rangle.
\label{eq:FW_with_large_lambda}
\end{align}

\end{enumerate}
\end{lem}

\begin{proof}
For $\lambda \geq 0$, it follows from the $L$-smoothness that
\begin{align}
f(x_{t} - \lambda d_{t}) \leq f(x_{t}) - \lambda \langle \nabla f(x_{t}), \, d_{t} \rangle + \frac{\lambda^{2}}{2} L \| d_{t} \|^{2}. 
\label{eq:L-smooth_with_d_t}
\end{align}
\begin{enumerate}
\item[(a)]
We begin with a FW step with $\lambda^{\ast}_{t} < 1$. 
By letting $\lambda = \lambda_{t}^{\ast}$ in the RHS of~\eqref{eq:L-smooth_with_d_t}, 
we have inequality~\eqref{eq:PW_FW_with_small_lambda} as follows.
\begin{align}
f(x_{t+1}) 
& \leq f(x_{t}) -  \frac{\langle \nabla f(x_{t}), \, d_{t} \rangle^{2}}{2L \| d_{t} \|^{2}}
\notag \\
& \leq f(x_{t}) -  \frac{\langle \nabla f(x_{t}), \, d_{t} \rangle^{2}}{2L D^{2}}. 
\label{eq:PW_FW_with_small_lambda_derived}
\end{align}
Next, consider a descent step with $\lambda^{\ast}_{t} < \Lambda_{t}^{\ast}$. By $\lambda_t = \argmin_{\lambda \in [0, \, \Lambda_{t}^{\ast}]} f(x_{t} - \lambda d_{t})$ and $\lambda^{\ast}_{t} < \Lambda_{t}^{\ast}$, we have $f(x_{t+1})\leq f(x_{t} - \lambda^{\ast}_{t} d_{t})$. By letting $\lambda = \lambda_{t}^{\ast}$ in the RHS of~\eqref{eq:L-smooth_with_d_t}, we derive the desired inequality.

Finally, consider a descent step with $\lambda^{\ast}_{t} \geq \Lambda_{t}^{\ast}$. 
Since step $t$ is not a drop step, 
$\Lambda_{t}^{\ast} > \lambda_{t}$ holds. 
Here $\lambda_{t}$ is a global minimizer of the convex function $f(x_{t} - \lambda d_{t})$. 
Therefore we have
\(
f(x_{t+1}) 
\leq f(x_{t} - \lambda^{\ast}_{t} d_{t})
\)
and this RHS is bounded by that of~\eqref{eq:PW_FW_with_small_lambda_derived} 
owing to~\eqref{eq:L-smooth_with_d_t}. 

\item[(b)]
The condition $\lambda_{t}^{\ast}=\frac{\langle \nabla f(x_{t}), \, d_{t} \rangle}{L\| d_{t} \|^{2}} \geq 1$ implies $\langle \nabla f(x_{t}), \, d_{t} \rangle \geq L\| d_{t} \|^{2}$. 
By letting $\lambda = 1$ in the RHS of~\eqref{eq:L-smooth_with_d_t}, 
we have inequality~\eqref{eq:FW_with_large_lambda} as follows.
\begin{align}
f(x_{t+1}) 
& \leq 
f(x_{t}) - \langle \nabla f(x_{t}), \, d_{t} \rangle + \frac{1}{2} L \| d_{t} \|^{2}
\notag \\
& \leq 
f(x_{t}) - \frac{1}{2} \langle \nabla f(x_{t}), \, d_{t} \rangle.
\notag
\end{align}
\end{enumerate}
\end{proof}

% Second lemma
\begin{lem}
\label{thm:2_nabla_f_at_wt}
For each step $t$, 
an inequality $2 \langle \nabla f(x_{t}), \, d_{t} \rangle \geq \langle \nabla f(x_{t}), \, a_{t} - w_{t} \rangle$ holds. 
\end{lem}

\begin{proof}
First, 
suppose that step $t$ is the pairwise step. 
By the definition of the algorithm, 
we have 
$\left< \nabla f(x_t), a_t -s_t \right> \geq \left< \nabla f(x_t), x_t -w_t \right>$. 
In addition, 
by the definitions of $x_{t}$ and $s_{t}$, 
we have
\begin{align}
\left< \nabla f (x_t), x_t \right> &\geq \left< \nabla f(x_t), s_t \right>.
\label{eq:nabla_x_t_vs_s_t}
\end{align}
Hence we have $\left< \nabla f(x_t), a_t -s_t \right> \geq \left< \nabla f(x_t), s_t -w_t \right>$ 
and by adding the LHS of this inequality to both sides we have
\begin{equation}\notag %\label{eq5-4}
2\left< \nabla f(x_t), a_t -s_t \right> \geq \left< \nabla f(x_t), a_t - w_t \right>. 
\end{equation}

Next, 
suppose that step $t$ is the Frank-Wolfe step. 
By the definition of the algorithm, 
we have $\left< \nabla f(x_t), x_t -w_t \right> > \left< \nabla f(x_t), a_t -s_t \right>$. 
It follows from this inequality and~\eqref{eq:nabla_x_t_vs_s_t} that
$\left< \nabla f(x_t), x_t -w_t \right> > \left< \nabla f(x_t), a_t - x_t \right>$.
By adding the LHS of this inequality to both sides we have 
\begin{equation}\notag %\label{eq5-5}
2\left< \nabla f(x_t), x_t -w_t \right> > \left< \nabla f(x_t), a_t - w_t \right>.
\end{equation}
\end{proof}

We are now in a position to prove 
Theorems~\ref{thm:mu_strong_linear_convergence} and~\ref{thm:only_Lsmooth_Tinv_convergence}. 
Let $h_{t} := f(x_{t}) - f(x^{\ast})$. 
Then, 
by the convexity of $f$ and the definition of $w_{t}$, 
we have 
\begin{align}
h_{t} 
\leq \langle \nabla f(x_{t}), \, x_{t} - x^{\ast} \rangle 
\leq \langle \nabla f(x_{t}), \, x_{t} - w_{t} \rangle. 
\label{eq:ht_leq_nabla_f_dt}
\end{align}

We start with the slightly more involved proof of Theorem~\ref{thm:mu_strong_linear_convergence}.

\begin{proof}[Proof of Theorem~\ref{thm:mu_strong_linear_convergence}]
First, 
we focus on the case that step $t$ is not a drop step, which is considered in Lemma~\ref{thm:not_drop_step}. 
Suppose that step $t$ is case (a) of Lemma~\ref{thm:not_drop_step}. 
By combining inequality~\eqref{eq:PW_FW_with_small_lambda}, 
Lemmas~\ref{thm:2_nabla_f_at_wt} and~\ref{thm:mu_strong_pyramidal_width}, 
we have
\begin{align}
h_{t} - h_{t+1} 
& \geq \frac{\langle \nabla f(x_{t}), \, d_{t} \rangle^{2}}{2LD^{2}} 
\geq \frac{\left< \nabla f(x_t), a_t - w_t \right>^{2}}{8LD^{2}}
\notag \\
& \geq \frac{\mu \delta^{2}}{4LD^{2}} h_{t}.
\label{eq:ht_ht1_geq_mu_delta_ht}
\end{align}
Then, consider case (b) of Lemma~\ref{thm:not_drop_step}, 
where $d_{t} = x_{t} - w_{t}$. 
By inequalities~\eqref{eq:FW_with_large_lambda} and~\eqref{eq:ht_leq_nabla_f_dt}, 
we have 
\begin{align}
h_{t} - h_{t+1} \geq \frac{1}{2} \langle \nabla f(x_{t}), \, d_{t} \rangle \geq \frac{1}{2} h_{t}.
\label{eq:ht_ht1_geq_half_ht}
\end{align}
Therefore 
by letting 
$\hat{c}_{f,P} = \min \{ \frac{1}{2}, \, \frac{\mu \delta^{2}}{4LD^{2}} \}$
we can deduce from~\eqref{eq:ht_ht1_geq_mu_delta_ht} and~\eqref{eq:ht_ht1_geq_half_ht} that 
\begin{align}
h_{t+1} \leq (1 - \hat{c}_{f, P}) h_{t}
\notag
\end{align}
if step $t$ is not a drop step. 

Next, we take the drop steps into account. 
If step $t$ is a drop step, it is clear that $h_{t+1} \leq h_{t}$. 
In addition, 
we bound the number of the drop steps in the algorithm. 
Let $T_{\mathrm{FW}}$, $T_{\mathrm{desc}}$, and $T_{\mathrm{drop}}$ be 
the numbers of the FW steps, descent steps, and drop steps, respectively. 
Note that $T = T_{\mathrm{FW}} + T_{\mathrm{desc}} + T_{\mathrm{drop}}$. 
Since $\# S_{T} \geq 1$, an inequality $T_{\mathrm{drop}} \leq T_{\mathrm{FW}}$ holds. 
Therefore we have
\begin{align}
T 
& = T_{\mathrm{FW}} + T_{\mathrm{desc}} + T_{\mathrm{drop}} 
\leq 2 T_{\mathrm{FW}} + T_{\mathrm{desc}}
\notag \\
& \leq 2 (T_{\mathrm{FW}} + T_{\mathrm{desc}}). 
\notag
\end{align}

Finally, by combining the above arguments, we have
\begin{align}
h_{T} 
& \leq (1 - \hat{c}_{f,P})^{T_{\mathrm{FW}} + T_{\mathrm{desc}}} h_{0}
\leq (1 - \hat{c}_{f,P})^{T/2} h_{0}
\notag \\
& \leq \exp \left( - \frac{\hat{c}_{f,P}}{2} T \right) h_{0}.
\notag
\end{align}
\end{proof}

Next, we provide the proof of Theorem~\ref{thm:only_Lsmooth_Tinv_convergence}. 
\begin{proof}[Proof of Theorem~\ref{thm:only_Lsmooth_Tinv_convergence}]
First, suppose that step $t$ is a descent step. By the algorithm of BPCG, it holds
$$ \left< \nabla f(x_t), d_t \right> \geq \left< \nabla f(x_t), x_t - w_t \right> .$$
Combining this with Lemma~\ref{thm:not_drop_step} (a) and $0\leq h_t \leq \left< \nabla f(x_t), x_t - w_t \right> $ ((\ref{eq:ht_leq_nabla_f_dt})), we have
\begin{align}
h_t - h_{t+1} &\geq  \frac{1}{2LD^2} \left< \nabla f(x_t), d_t \right>^2  \notag \\
&\geq  \frac{1}{2LD^2} \left< \nabla f(x_t), x_t - w_t \right> ^2 \notag \\
&\geq \frac{1}{2LD^2} h_t ^2. \label{descent_steps_bound_thm3.1}
\end{align}
Next consider the case where step $t$ is a FW step. First, by the 
$L$-smoothness of $f$ and (\ref{eq:ht_leq_nabla_f_dt}), we have
\begin{align*} f(x_{t} - \lambda (x_{t} - w_t) ) 
&\leq f(x_{t}) - \lambda \left< \nabla f(x_{t}), x_t - w_t \right>
+\frac{L}{2} \lambda ^2 \| x_t  - w_t \|^2\\
&\leq f(x_{t}) - \lambda \left< \nabla f(x_{t}), x_t - w_t \right> + \frac{\lambda^{2}}{2} L D^{2}  \\
&\leq f(x_{t}) - \lambda h_t + \frac{\lambda^{2}}{2} L D^{2}
\end{align*}
for $\lambda \geq 0$.
Subtracting $f(x^{\ast})$ from both sides, we get
\begin{equation}\label{fw_Lsmooth_bound}
f(x_{t} - \lambda (x_{t} - w_t) ) - f(x^{\ast}) \leq h_t - \lambda h_t + \frac{\lambda^{2}}{2} L D^{2}.
\end{equation}
Consider the following two cases:
\begin{description}
\item[(i) $h_t \leq LD^2$]
\ \\
By the definition of $\lambda_t$, 
\begin{equation}\label{stepsize_ineq}
 h_{t+1} = f(x_t - \lambda_t(x_t - w_t)) - f(x^{\ast})\leq  f(x_t - \lambda(x_t - w_t)) - f(x^{\ast})
\end{equation} 
for $\lambda \in [0, 1]$. Using  (\ref{fw_Lsmooth_bound}) and (\ref{stepsize_ineq}) for $\lambda=\frac{h_t}{LD^2}\leq 1$, we have 
\begin{equation}\label{first_inequality_fw_step_3.1}
 h_{t+1}\leq f\left(x_t -\frac{h_t}{LD^2}(x_t - w_t)\right) - f(x^{\ast})\leq   h_t - \frac{h_t ^2}{2LD^2}.
 \end{equation}
\item[(ii) $h_t \geq LD^2$]
\ \\
Combining (\ref{fw_Lsmooth_bound}) with (\ref{stepsize_ineq}) for $\lambda=1$, we have
\begin{equation}\label{second_inequality_fw_step_3.1}
h_{t+1} \leq f\left(x_t - 1\cdot(x_t - w_t)\right) - f(x^{\ast})\leq \frac{LD^2}{2}.
\end{equation}
\end{description}
By (\ref{first_inequality_fw_step_3.1}) and (\ref{second_inequality_fw_step_3.1}), we have
\begin{equation}\label{fw_steps_bound_thm3.1}
h_{t+1}\leq 
\begin{cases}
h_t - \frac{h_t ^2}{2LD^2} &(h_t \leq LD^2),\\
\frac{LD^2}{2}\leq \frac{h_t}{2} &(\mathrm{otherwise}).
\end{cases}
\end{equation}

For a iteration $T$, we define $T_{\mathrm{FW}}, T_{\mathrm{desc}}$ and $T_{drop}$ in the same way as the proof of Theorem~\ref{thm:mu_strong_linear_convergence}.  Using (\ref{descent_steps_bound_thm3.1}) and (\ref{fw_steps_bound_thm3.1}), we can show 
\begin{equation}\label{desc_fw_bound}
h_{T} \leq \frac{2LD^2}{T_{\mathrm{desc}}+T_{\mathrm{FW}}}.
\end{equation}
This can be shown just in the same way as the proof of Corollary 4.2 in \cite{pok18bcg}; the value $4 L_f$ in the proof is replaced by $2LD^2$ in this case.

Finally, as shown in the proof of Theorem~\ref{thm:mu_strong_linear_convergence}, $T \leq 2(T_{\mathrm{desc}}+T_{\mathrm{FW}})$ holds. By substituting this to (\ref{desc_fw_bound}), we have
$$ h_{T} \leq \frac{4LD^2}{T} .$$
\end{proof}

Although  we showed Theorem~\ref{thm:only_Lsmooth_Tinv_convergence} in the Euclidean case for the simplicity of the argument,  it is easy to see from the proof of Theorem~\ref{thm:only_Lsmooth_Tinv_convergence} that the convergence guarantee is also applicable to the infinite-dimensional general convex feasible regions. 

As shown in Theorem~\ref{thm:only_Lsmooth_Tinv_convergence} and Theorem~\ref{thm:mu_strong_linear_convergence}, the BPCG algorithm achieves a convergence rate which is no worse than that of the PCG. In addition, since the BPCG does not exhibit swap steps,
the convergence rate of BPCG including (possibly dimension-dependent) constant factors is considered to be better than that of the PCG especially when the dimension of the feasible region is high: PCG's convergence rate (see \citep{lacoste15}) is of the form $h_t \leq h_0 \exp(-\rho k(t))$, where $k(t) \geq t / (3 |\mathcal A|! +1)$, where $\mathcal A$ is the set of vertices generating the polytope. Even for $0/1$-polytopes in $n$-dimensional space, this can be as bad as $2^n$ and the factor $(2^n)!$ is even worse.  In fact, this dimension-dependence in the rate is also the reason the convergence proof of PCG does not generalize to infinite-dimensional cases. As such BPCG's convergence rate is much more in line with that of the Away-Step Frank-Wolfe algorithm (see \citep{lacoste15}). Moreover, since the iterations of BPCG include many local updates in which no new atoms are added, it is expected that the BPCG algorithm outputs sparser solutions than the PCG algorithm in terms of the support size of the supporting convex combination. 

\begin{rem}
The constant factor $\frac{1}{2LD^2}$ of the  bound (\ref{eq:PW_FW_with_small_lambda}) in Lemma~\ref{thm:not_drop_step} does not depend on the dimension of feasible domains, which is important to guarantee $O(\frac{1}{T})$ convergence in infinite-dimensional cases. The similar BCG algorithm in \cite{pok18bcg} employs Simplex Gradient Descent (SiGD) instead of local PCG steps. However, the lower bound for the progress of SiGD includes a dimension-dependent term and we cannot guarantee  $O(\frac{1}{T})$ convergence in infinite-dimensional cases in general for BCG.
\end{rem}

\subsection{Lazified version of BPCG}\label{subsection_lazified}

We can consider the lazified version (see \cite{pok17lazy,BPZ2017jour}) of the BPCG
as shown by Algorithm~\ref{alg:lazifiedBPCG}. It employs the estimated Frank-Wolfe gap $\Phi_t$ instead of computing the Frank-Wolfe gap in each iteration. The lazification technique helps reduce the number of access to LMOs, which improves the computational efficiency of Algorithm~\ref{alg:lazifiedBPCG} since we only need to access the active set $S_t$ when the pairwise gap $\left< \nabla f(x_t), a_t - s_t \right>$ is larger than $\Phi_t$.

\setcounter{algorithm}{1}
\begin{algorithm}[H]
%\begin{algorithm}[!h]
\caption{Lazified BPCG}
\label{alg:lazifiedBPCG}
\begin{algorithmic}[1]
\REQUIRE convex smooth function $f$, start vertex $x_0 \in V(P)$, accuracy $J \geq 1$.  
\ENSURE points $x_{1}, \ldots, x_T$ in $P$ 
\STATE $\Phi_0 \leftarrow \max_{v\in P} \langle \nabla f(x_0) , \, x_0 - v \rangle/2$
\STATE $S_{0} \leftarrow \{ x_0 \}$
\FOR {$t = 0$ to $T-1$}
	\STATE $a_t \leftarrow \argmax_{v \in S_t} \left< \nabla f(x_t), v \right>$ \quad \COMMENT {away vertex}
	\STATE $s_t \leftarrow \argmin_{v \in S_t} \left< \nabla f(x_t), v \right>$ \quad \COMMENT {local FW}
	\IF {$\left< \nabla f(x_t), a_t -s_t \right> \geq \Phi_t$ } \label{st:if_local}
		\STATE $d_t = a_t - s_t$
		\STATE $\Lambda_{t}^{\ast} \leftarrow c[x_{t}](a_{t})$
		\STATE $\lambda_t \leftarrow \argmin_{\lambda \in [0, \, \Lambda_{t}^{\ast}]} f(x_{t} - \lambda d_{t})$ %\quad \COMMENT {line search of a pairwise step}
		\STATE $x_{t+1} \leftarrow x_t - \lambda_t d_t$
		\STATE $\Phi_{t+1} \leftarrow \Phi_{t}$
		\IF {$\lambda_{t} < \Lambda_{t}^{\ast}$}
			\STATE $S_{t+1} \leftarrow S_{t}$ \quad \COMMENT {descent step}
		\ELSE
			\STATE $S_{t+1} \leftarrow S_{t} \setminus \{ a_{t} \}$
			\quad \COMMENT {drop step}
		\ENDIF
	\ELSE 
		\STATE $w_t \leftarrow \argmin_{v \in V(P)} \left< \nabla f(x_t), v \right>$ \COMMENT {global FW} \label{st:LMO_for_gFW}
		\IF{$\left< \nabla f(x_t), x_t - w_t\right> \geq \Phi_t /J$} \label{gap_fw_condition}
			\STATE $d_t = x_t - w_t$ 
			\STATE $\lambda_t  \leftarrow \argmin_{\lambda \in [0, 1]} f(x_t - \lambda d_t)$ %\quad \COMMENT {line search of an FW step} 
			\STATE $x_{t+1} \leftarrow x_t - \lambda_t d_t$
			\STATE $\Phi_{t+1} \leftarrow \Phi_{t}$
			\STATE $S_{t+1} \leftarrow S_t \cup \{ w_t \}$
			\quad \COMMENT {FW step}
		\ELSE
			\STATE $x_{t+1} \leftarrow x_{t}$
			\STATE $\Phi_{t+1} \leftarrow \Phi_t /2$
			\STATE $S_{t+1} \leftarrow S_{t}$
			\quad \COMMENT {gap step}
		\ENDIF
	\ENDIF
\ENDFOR
\end{algorithmic}
\end{algorithm}

The theoretical analysis of the lazified BPCG can be done in the almost same way as  the proof of Theorem 3.1 in \cite{pok18bcg} for the strongly convex case; the general smooth case follows similarly. For the descent step, the analysis differs, but we can use Lemma~\ref{thm:not_drop_step}. For the only $L$-smooth case, we can show in a similar as the proof of the strongly convex case.
As a result, 
we can show the following theorem. 

\begin{thm}
\label{thm:lazifiedBPCG}
Let $P$ be a convex feasible domain with diameter $D$.
Furthermore, let $f$ be a $L$-smooth convex function and consider the sequence $\{ x_{i} \}_{i=0}^{T} \subset P$ obtained from the lazified BPCG algorithm (Algorithm~\ref{alg:lazifiedBPCG}). 
\begin{description}
\item[Case (A)]
If $f$ is $\mu$-strongly convex and $P$ is a polytope with pyramidal width $\delta > 0$, we have
\begin{align}
f(x_{T}) - f(x^{\ast}) = O \left( \exp \left( -c\, T \right) \right)
\quad
(T \to \infty)
\notag 
\end{align}
for a constant $c > 0$ independent of $T$. 
\item[Case (B)]
If $f$ is only convex and $L$-smooth, 
we have
\begin{align}
f(x_{T}) - f(x^{\ast}) = O \left( \frac{1}{T} \right)
\quad
(T \to \infty).
\notag
\end{align}
\end{description}
\end{thm}

\begin{proof}
The proof tracks that of \cite{pok18bcg}, especially for Case(A).

We first consider Case (B). In the same way as \cite{pok18bcg}, we divide the iteration into sequences of epochs that are demarcated by the \emph{gap steps} in which the value $\Phi_t$ is halved. We bound the number of iterations in each epoch.

By the convexity of $f$ and the definition of $w_t$, we have
$$f(x_t) - f(x^{\ast}) \leq \left<\nabla f(x_t), x_t - x^{\ast} \right> \leq \left<\nabla f(x_t), x_t -w_t\right>  .$$
If iteration $t-1$ is a gap step, it holds 

\begin{equation}\label{primal_gap_bound}
f(x_t) - f(x^{\ast}) \leq \left<\nabla f(x_t), x_t -w_t \right> \leq  \frac{2 \Phi_t}{J} \leq  2 \Phi_t .
\end{equation}

We note that (\ref{primal_gap_bound}) also holds at $t=0$ by the definition of $\Phi_0$. By (\ref{primal_gap_bound}), if $2\Phi_t \leq \epsilon$ holds for $\epsilon > 0$,  the primal gap $f(x_t) - f(x^{\ast})$ is upper bounded by $\epsilon$.  Therefore, the total number of epochs $N_{\Phi}$  to achieve $f(x_t) - f(x^{\ast}) \leq \epsilon$ is bounded in the following way:
\begin{equation}\label{num_epochs}
 N_{\Phi}  \leq \left\lceil \log \frac{2\Phi_0}{\epsilon}  \right\rceil .
 \end{equation}

Next, we consider the epoch that starts from iteration $t$ and use the notation $u$ to index the iterations within the epoch. We note  that $\Phi_t = \Phi_u$ within the epoch. 

We divide each iteration into three cases according to types of steps. First, consider the case $u$ iteration is a FW step that means $d_u = x_u - w_u$. Using  Lemma~\ref{thm:not_drop_step} and the condition $\left<\nabla f(x_u), x_u -w_u \right> \geq  \Phi_u /J = \Phi_t /J$, we have 
\begin{align}
f(x_u) - f(x_{u+1}) &\geq \min\left\{ \frac{\left< \nabla f(x_u),x_u - w_u \right> ^2}{2LD^2}, \frac{1}{2}\left< \nabla f(x_u),x_u - w_u \right>  \right\} \notag\\
&\geq \frac{\Phi_t}{2J} \min\left\{ 1, \frac{\Phi_t}{LD^2 J} .\right\} . \label{fw_step_bound}
\end{align}

Next, consider the case $u$ iteration is a descent step that means $d_u = a_u - s_u$ and $\lambda_u < \Lambda_u ^\ast$. Using Lemma~\ref{thm:not_drop_step} (a) and the inequality $\left<\nabla f(x_u), a_u - s_u \right> \geq  \Phi_u = \Phi_t$, we have 

\begin{align}
f(x_u) - f(x_{u+1}) &\geq \frac{\left< \nabla f(x_u), d_u\right> ^2}{2LD^2} \notag\\
&\geq \frac{\Phi_t ^2}{2LD^2} \label{des_step_bound}.
\end{align}

Finally, consider the case $u$ iteration is a drop step that means $d_u = a_u - s_u$ and $\lambda_u = \Lambda_u ^\ast$. In this case, it holds
\begin{equation}\label{drop_step_bound}
f(x_u) - f(x_{u+1}) \geq 0 .
\end{equation}

We bound the total number of iterations $T$ to achieve $f(x_T)-f(x^{\ast}) \leq \epsilon$. Let $N_{\mathrm{FW}}, N_{\mathrm{desc}}, N_{\mathrm{drop}}, N_{\mathrm{gap}}$ be the number of FW steps, descent steps, drop steps and gap steps respectively. In addition, we denote the number of FW steps, descent steps in epoch $t$ by $N_{\mathrm{FW}} ^t, N_{\mathrm{desc}} ^t $  respectively. Using $N_{\mathrm{drop}} \leq N_{\mathrm{FW}}$ and (\ref{num_epochs}), we have

\begin{equation}\label{iteration_bound}
T \leq N_{\mathrm{FW}}+ N_{\mathrm{desc}}+N_{\mathrm{drop}}+ N_{\mathrm{gap}} \leq \left\lceil \log \frac{2\Phi_0}{\epsilon}  \right\rceil + 2N_{\mathrm{FW}}+ N_{\mathrm{desc}}\leq \left\lceil \log \frac{2\Phi_0}{\epsilon}  \right\rceil + \sum_{t : \mathrm{epoch}} (2N_{\mathrm{FW}} ^t + N_{\mathrm{desc}} ^t)  .
\end{equation}
Here, we bound $2N_{\mathrm{FW}} ^t + N_{\mathrm{desc}} ^t$ in epoch $t$.

 Let $t'$ be the index of iteration where epoch $t$ starts. We consider the following two cases: 

\begin{description}
\item[(I) $\Phi_t \geq LD^2 J$]
\ \\
By (\ref{primal_gap_bound}),(\ref{fw_step_bound}), (\ref{des_step_bound}) and the condition $\Phi_t \geq LD^2 J$, we have 
\begin{align}
2\Phi_t &\geq f(x_{t'}) - f(x^{\ast})  \notag \\
&\geq  N_{\mathrm{FW}} ^t \cdot \frac{\Phi_t}{2J} +  N_{\mathrm{desc}} ^t \cdot \frac{\Phi_t ^2}{2LD^2} \notag \\
&\geq 2 N_{\mathrm{FW}} ^t \cdot \frac{\Phi_t}{4J} + N_{\mathrm{desc}} ^t \cdot \frac{\Phi_t}{2LD^2} \cdot LD^2J \notag\\
&= \Phi_t \left(2 N_{\mathrm{FW}} ^t \cdot \frac{1}{4J} +  N_{\mathrm{desc}} ^t \cdot \frac{J}{2}  \right)\notag .
\end{align}
Thus, 
\begin{equation}\label{iteration_bound_A}
2 N_{\mathrm{FW}} ^t + N_{\mathrm{desc}} ^t \leq \max\{ 8J, \frac{4}{J} \}.
\end{equation}

\item[(II) $\Phi_t <  LD^2 J$]
\ \\
Using (\ref{primal_gap_bound}),(\ref{fw_step_bound}) and (\ref{des_step_bound}), we have 

$$2\Phi_t \geq f(x_{t'}) - f(x^{\ast}) \geq N_{\mathrm{FW}} ^t \cdot \frac{\Phi_t ^2}{2LD^2 J^2}  + N_{\mathrm{desc}} ^t \cdot \frac{\Phi_t ^2}{2LD^2} = \Phi_t ^2 \left( 2 N_{\mathrm{FW}} ^t \cdot \frac{1}{4LD^2J^2} + N_{\mathrm{desc}} ^t \cdot \frac{1}{2LD^2} \right).$$
Thus, we have 
\begin{equation}\label{non_sc_iter_bound_B}
2 N_{\mathrm{FW}} ^t + N_{\mathrm{desc}} ^t  \leq \frac{1}{\Phi_t} \max\{8LD^2J^2, 4LD^2 \}.
\end{equation}
\end{description}
Since $\Phi_t$ dose not increase and (\ref{iteration_bound_A}) holds, the number of iterations in which $\Phi_t \geq LD^2 J$ holds is bounded. We define $t_1$ as the first epoch where $\Phi_t < LD^2 J$ is satisfied and $T_1$ as the total number of the iterations by epoch $t_1$. Then, by (\ref{iteration_bound}) and (\ref{non_sc_iter_bound_B}), we have
$$
T \leq T_1 + \left\lceil \log \frac{2\Phi_0}{\epsilon}  \right\rceil + \sum_{t\geq t_1} (2N_{\mathrm{FW}} ^t + N_{\mathrm{desc}} ^t) 
\leq T_1 + \left\lceil \log \frac{2\Phi_0}{\epsilon}  \right\rceil + \sum_{t : \mathrm{epoch}} \frac{C_1}{\Phi_t} 
,$$
where $C_1=\max\{8LD^2J^2, 4LD^2 \} $.
In addition, it holds that
$$
 \sum_{t : \mathrm{epoch}} \frac{C_1}{\Phi_t} \leq \frac{C_1}{\Phi_0} \sum_{n=0}^{N_{\Phi}} 2^n \leq \frac{C_1}{\Phi_0} (2^{N_{\Phi}+1} -1) \leq \frac{C_1}{\Phi_0} 2^{\log \frac{2\Phi_0}{\epsilon}+2 }
 = \frac{8C_1}{\epsilon}.
$$
Therefore , we have
$$ T \leq T_1 +\left\lceil \log \frac{2\Phi_0}{\epsilon}  \right\rceil + \frac{8C_1}{\epsilon}.$$
Therefore we drive $\epsilon= O(\frac{1}{T})$.

Next, consider Case (A). The argument by (\ref{iteration_bound}) can be done exactly the same as Case (B).  In the same way as Case (B), we consider the two cases (I) $\Phi_t \geq LD^2 J$ and (II) $\Phi_t <  LD^2 J$. For the case (I), we can derive just the same result (\ref{iteration_bound_A}). For the case (II), we analyze in a different way from Case (B) in the following way:

\begin{description}
\item[(II-A) $\Phi_t <  LD^2 J$]
\ \\ 
If iteration $t-1$ is a gap step, using the same argument as \cite{pok18bcg}, we have 
\begin{equation}\label{sc_primal_gap}
f(x_t) - f(x^{\ast}) \leq \frac{8\Phi_t ^2}{\mu}.
\end{equation}
Using (\ref{sc_primal_gap}), we drive the following bound for Case (A);
\begin{align*}
\frac{8\Phi_t ^2}{\mu} & \geq f(x_{t'}) - f(x^{\ast}) \\
&\geq N_{\mathrm{FW}} ^t \cdot \frac{\Phi_t ^2}{2LD^2 J^2} + N_{\mathrm{desc}} ^t \cdot \frac{\Phi_t ^2}{2LD^2} \\
&\geq\Phi_t ^2 \left(2 N_{\mathrm{FW}} ^t \cdot \frac{1}{4LD^2J^2} + N_{\mathrm{desc}} ^t \cdot \frac{1}{2LD^2} \right) .
\end{align*}
Therefore, we have 
\begin{equation}\label{sc_iter_bound_A}
2 N_{\mathrm{FW}} ^t + N_{\mathrm{desc}} ^t  \leq \frac{8}{\mu} \max\{ 4LD^2J^2, 2LD^2 \}. 
\end{equation}
\end{description}

By (\ref{iteration_bound}), (\ref{iteration_bound_A}) and (\ref{sc_iter_bound_A}), it holds
\begin{equation}
T \leq C_2 \left\lceil \log \frac{2\Phi_0}{\epsilon}  \right\rceil,
\end{equation}
where $C_2 = 1+\max\{ \max\{ 8J, \frac{4}{J}\},\frac{8}{\mu} \max\{ 4LD^2J^2, 2LD^2 \} \} $. Thus, we derived the desired result for Case (A).
\end{proof}

As well as Theorem~\ref{thm:only_Lsmooth_Tinv_convergence}, it is easily to see from the proof that the convergence guarantee of Case (B) in Theorem~\ref{thm:lazifiedBPCG} can be also applicable to infinite-dimensional general convex feasible domains.

\subsection{Finer sparsity control in BPCG}
\label{sec:sparsityControl}

In this section we explain how the sparsity of the BPCG algorithm can be further controlled while changing the convergence rate only by a small constant factor. To this end we modify the step selection condition in Line~\ref{eq:select} in Algorithm~\ref{alg:BPCG} to incorporate a scaling factor $K_{\mathrm{sc}} \geq 1.0$: 

\begin{equation}
% \tag{relaxCond}
K_{\mathrm{sc}} \cdot \langle \nabla f(x_t), a_t - s_t \rangle \geq \langle \nabla f(x_t), x_t - w_t \rangle.
\end{equation}

With this modification Lemma~\autoref{thm:2_nabla_f_at_wt} changes as follows

\begin{lem}[Modified version of Lemma~\autoref{thm:2_nabla_f_at_wt}]
For each step $t$ in Algorithm~\ref{alg:BPCG}, an inequality 
$$
(K_{\mathrm{sc}}+1) \cdot \langle \nabla f(x_t), d_t \rangle \geq \langle \nabla f(x_t), a_t - w_t \rangle
$$
holds.
\begin{proof}% [Proof Sketch]
If we take a pairwise step we have 
$$
K_{\mathrm{sc}} \cdot \langle \nabla f(x_t), a_t - s_t \rangle \geq \langle \nabla f(x_t), x_t - w_t \rangle.
$$
Moreover, it holds as before 
\begin{equation}
\label{eq:convexComb}
%\tag{convexComb}
\langle \nabla f(x_t), x_t \rangle \geq \langle \nabla f(x_t), s_t\rangle,
\end{equation}
so that we obtain
$$
K_{\mathrm{sc}} \cdot \langle \nabla f(x_t), a_t - s_t \rangle \geq \langle \nabla f(x_t), s_t - w_t \rangle.
$$
Now we add $\langle \nabla f(x_t), a_t - s_t \rangle$ to both sides of this inequality and obtain:
$$
(K_{\mathrm{sc}}+1) \cdot \langle \nabla f(x_t), a_t - s_t \rangle \geq \langle \nabla f(x_t), a_t - w_t \rangle,
$$
which is the claim in this case as $d_t = a_t - s_t$.

In case we took a normal FW step it holds:
$$
K_{\mathrm{sc}} \cdot \langle \nabla f(x_t), a_t - s_t \rangle < \langle \nabla f(x_t), x_t - w_t \rangle,
$$
and together with \eqref{eq:convexComb}
$$
K_{\mathrm{sc}} \cdot \langle \nabla f(x_t), a_t - x_t \rangle < \langle \nabla f(x_t), x_t - w_t \rangle.
$$
Now adding $K_{\mathrm{sc}} \cdot \langle \nabla f(x_t), x_t - w_t \rangle$ to both sides we obtain
$$
K_{\mathrm{sc}} \cdot \langle \nabla f(x_t), a_t - w_t \rangle < (K_{\mathrm{sc}}+1) \cdot \langle \nabla f(x_t), x_t - w_t \rangle,
$$
and hence 
$$
\langle \nabla f(x_t), a_t - w_t \rangle < \frac{(K_{\mathrm{sc}}+1)}{K_{\mathrm{sc}}} \cdot \langle \nabla f(x_t), x_t - w_t \rangle \leq (K_{\mathrm{sc}}+1) \cdot \langle \nabla f(x_t), x_t - w_t \rangle,
$$
as required as $d_t = x_t - w_t$.
\end{proof}
\end{lem}

Plugging this modified lemma back into the rest of the proof of Theorem \ref{thm:mu_strong_linear_convergence} leads to a small constant factor loss of $(K_{\mathrm{sc}}+1)^2 / 4$ in the convergence rate, which stems from the simple fact that progress from smoothness is proportional to the squared dual gap estimate, which weakened by a factor of $(K_{\mathrm{sc}}+1)/2$. 

\begin{rem}[Sparsity BPCG vs. lazy BPCG]
Observe that the non-lazy BPCG is sparser than the lazy BPCG. While counter-intuitive, the optimization for the local active set maximizes sparsity already and the non-lazy variant uses tighter $\Phi$ bounds as they are updated in each iteration promoting prolonged optimization over active set before adding new vertices. Although it can happen that the lazy BPCG gets sparser solutions than the non-lazy BPCG, this observation may help to understand  the behavior of the algorithms. 
\end{rem}

% Application to kernel quadrature
\section{Application to kernel quadrature}

Kernel herding is a well-known method for kernel quadrature and 
can be regarded as a Frank-Wolfe method on a reproducing kernel Hilbert space (RKHS). In the following we will use the BPCG algorithm for kernel herding with aim to exploit its sparsity. To describe this application, 
we consider the following notation for kernel quadrature on an RHKS. 
Let $\Omega \subset \mathbf{R}^{d}$ be a compact region and 
let a continuous function $K: \Omega \times \Omega \to \mathbf{R}$ be a symmetric positive-definite kernel. 
Then, 
the RKHS $\mathcal{H}_{K}(\Omega)$ for $K$ on $\Omega$ is uniquely determined. We denote the inner product in $\mathcal{H}_{K}(\Omega)$ by $\left<\cdot, \cdot \right>_K$.
We consider integration of a function $f \in \mathcal{H}_{K}(\Omega)$ 
with respect to a probability measure on $\Omega$. 
By letting $\mathscr{M}^{+}(1)$ be a set of probability measures on $\Omega$, 
we consider
\(
I = \int_{\Omega} f(x) d\mu(x)
\)
for $\mu \in \mathscr{M}^{+}(1)$. 
To approximate $I$, 
we consider a quadrature formula 
$I_{X, W} = \sum_{x \in X} w_{x} \, f(x)$ 
by taking finite sets $X \subset \Omega$ and $W = \{ w_{x} \}_{x \in X} \subset [0,1]$ with $\sum_{x \in X} w_{x} = 1$. 
The formula $I_{X}$ can be regarded as an integral of $f$ 
by the discrete probability measure $\xi = \sum_{x \in X} w_{x} \, \delta_{x} \in \mathscr{M}^{+}(1)$, 
where $\delta_{x}$ is the Dirac measure for $x \in X$. 
According to the reproducing property of the RKHS, 
the Dirac measure $\delta_{x}$ corresponds to the function $K(x, \cdot) \in \mathcal{H}_{K}(\Omega)$. 
The accuracy of the formula $I_{X, W}$ is estimated by 
the maximum mean discrepancy (MMD) 
$\gamma_{K}(\xi, \mu) := \mathscr{E}_{K}(\xi - \mu)^{1/2}$, where
\begin{align}
\mathscr{E}_{K}(\eta) := \int_{\Omega} \int_{\Omega} K(x, y) \, \mathrm{d}\eta(x) \, \mathrm{d}\eta(y)
\notag
\end{align}
is the $K$-energy of a signed measure $\eta$. 

To construct a good discrete measure $\xi$, 
we consider minimization of the squared MMD $F_{K, \mu}(\nu) := \gamma_{K}(\nu, \mu)^{2}$ as a function of $\nu \in \mathscr{M}^{+}(1)$.
Now, kernel herding is essentially running a Frank-Wolfe method to minimize $F_{K, \mu}$ over $\mathscr{M}^{+}(1)$, in which we start with a Dirac measure $\xi = \delta_{x_{0}}$ and 
sequentially add new Dirac measures $\delta_{x_{1}}, \delta_{x_{2}}, \ldots$ and form iterates as convex combinations. 
As such, we can naturally apply BPCG to the kernel herding. Here we note that 
the function $F_{K, \mu}$ is a $2$-smooth and convex function on $\mathscr{M}^{+}(1)$. 
In particular, we have
\begin{align}
& F_{K, \mu}(\theta + \alpha (\eta - \zeta)) 
\notag \\
& = 
F_{K, \mu}(\theta) + \alpha \left< \nabla F_{K, \mu}(\theta) , \, \eta - \zeta \right>_{K} + \alpha^{2} \, \mathscr{E}_{K}(\eta - \zeta)
\label{eq:f_2_smooth}
\end{align}
for any $\alpha \in \mathbf{R}$ and $\theta, \eta, \zeta \in \mathscr{M}^{+}(1)$ 
such that $\theta + \alpha (\eta - \zeta) \in \mathscr{M}^{+}(1)$. 
In addition, 
if $K(x,x) = 1$ and $K(x,y) \geq 0$ for any $x,y \in \Omega$, 
we have $\mathscr{E}_{K}(\eta - \zeta) \leq 2$ for any $\eta, \zeta \in \mathscr{M}^{+}(1)$. 
Therefore the term $\mathscr{E}_{K}(\eta - \zeta)$ is bounded in \eqref{eq:f_2_smooth}. 
This fact means that the diameter of $\mathscr{M}^{+}(1)$ with respect to $\mathscr{E}_{K}$ is bounded.

We note that in this problem setting,  the function $\int_{\Omega} K(x, \cdot) \ \mathrm{d}\mu(x)$ which is the embedding of $\mu$ to $\mathcal{H}_K$ satisfies  $\int_{\Omega} K(x, \cdot) \ \mathrm{d}\mu(x)  \in \overline{\mathrm{conv}(\{ K(x, \cdot) \mid x \in \Omega \}) } $, where the closure is taken with respect to the norm of $\mathcal{H}_K$ (see \cite{tsuji2021acceleration}).

In Algorithm~\ref{alg:BPCGforKH} we describe the BPCG algorithm applied to kernel herding. In addition, we present its lazified version Algorithm~\ref{alg:LazifiedBPCGforKH}. In these algorithms, we can regard the Dirac measures $\delta_{x}$ (i.e., $K(x, \cdot)$) as vertices or atoms in the RKHS $\mathcal{H}_{K}(\Omega)$. We assume the existence of the solution in line \autoref{KH_FW_LMO} in Algorithm~\ref{alg:BPCGforKH}. 

\setcounter{algorithm}{2}
%\begin{algorithm}[H]
\begin{algorithm}[ht]
%\begin{algorithm}
\caption{BPCG algorithm for kernel herding}
\label{alg:BPCGforKH}
\begin{algorithmic}[1]
\REQUIRE the function $F_{K, \mu}$, \ start measure $\xi_{0} \in \mathscr{M}^{+}(1)$ with $\mathop{\mathrm{supp}} \xi_{0} = \{ x_{0} \}$
\ENSURE discrete measures $\xi_{1}, \ldots, \xi_{T} \in \mathscr{M}^{+}(1)$
\STATE $X_{0} \leftarrow \mathop{\mathrm{supp}} \xi_{0} $
\FOR {$t = 0$ to $T-1$}
	\STATE $x_{t}^{\mathrm{A}} \leftarrow \argmax_{x \in X_{t}} \, \left< \nabla F_{K, \mu}(\xi_{t}), \,\delta_{x} \right>_{K}$
	\STATE $x_{t}^{\mathrm{S}} \leftarrow \argmin_{x \in X_{t}} \, \left< \nabla F_{K, \mu}(\xi_{t}), \, \delta_{x} \right>_{K}$
	\STATE $x_{t}^{\mathrm{W}} \leftarrow \argmin_{x \in \Omega} \, \left< \nabla F_{K, \mu}(\xi_{t}), \, \delta_{x} \right>_{K}$ 
	\label{KH_FW_LMO}
	\IF {$\left< \nabla F_{K, \mu}(\xi_{t}), \, \delta_{x_{t}^{\mathrm{A}}} - \delta_{x_{t}^{\mathrm{S}}} \right>_{K} \geq \left< \nabla F_{K, \mu}(\xi_{t}), \, \xi_{t} - \delta_{x_{t}^{\mathrm{W}}} \right>_{K} $}
		\STATE $\eta_{t} \leftarrow \delta_{x_{t}^{\mathrm{A}}} - \delta_{x_{t}^{\mathrm{S}}}$
		\STATE $\alpha_{t} \leftarrow \argmin_{\alpha \in [0, \, \xi_{t}(\{ x_{t}^{\mathrm{A}} \})]} F_{K, \mu}( \xi_{t} - \alpha \, \eta_{t})$
		\IF {$\alpha_{t} < \xi_{t}(\{ x_{t}^{\mathrm{A}} \})$}
			\STATE $X_{t+1} \leftarrow X_{t}$ 
		\ELSE
			\STATE $X_{t+1} \leftarrow X_{t} \setminus \{ x_{t}^{\mathrm{A}} \}$
		\ENDIF
	\ELSE
			\STATE $\eta_{t} \leftarrow \xi_{t} - \delta_{x_{t}^{\mathrm{W}}}$
			\STATE $\alpha_{t} \leftarrow \argmin_{\alpha \in [0,1]} F_{K, \mu}( \xi_{t} - \alpha \eta_{t} )$
			\STATE $X_{t+1} \leftarrow X_{t} \cup \{ x_{t}^{\mathrm{W}} \}$ 
	\ENDIF
	\STATE $\xi_{t+1} \leftarrow \xi_{t} - \alpha_{t} \eta_{t} $ 
\ENDFOR
\end{algorithmic}
\end{algorithm}

\setcounter{algorithm}{3}
\begin{algorithm}[H]
\caption{Lazified BPCG algorithm for kernel herding}
\label{alg:LazifiedBPCGforKH}
\begin{algorithmic}[1]
\REQUIRE the function $F_{K, \mu}$, \ start measure $\xi_{0} \in \mathscr{M}^{+}(1)$ with $\mathop{\mathrm{supp}} \xi_{0} = \{ x_{0} \}$, accuracy $J \geq 1$
\ENSURE discrete measures $\xi_{1}, \ldots, \xi_{T} \in \mathscr{M}^{+}(1)$
\STATE $\Phi_{0} \leftarrow \max_{x \in \Omega} \, \left< \nabla F_{K, \mu}(\xi_{0}) , \, \xi_{0} - \delta_{x} \right>_{K}/2$
\STATE $X_{0} \leftarrow \mathop{\mathrm{supp}} \xi_{0} $
\FOR {$t = 0$ to $T-1$}
	\STATE $x_{t}^{\mathrm{A}} \leftarrow \argmax_{x \in X_{t}} \, \left< \nabla F_{K, \mu}(\xi_{t}), \,\delta_{x} \right>_{K}$
	\STATE $x_{t}^{\mathrm{S}} \leftarrow \argmin_{x \in X_{t}} \, \left< \nabla F_{K, \mu}(\xi_{t}), \, \delta_{x} \right>_{K}$
	\IF {$\left< \nabla F_{K, \mu}(\xi_{t}), \, \delta_{x_{t}^{\mathrm{A}}} - \delta_{x_{t}^{\mathrm{S}}} \right>_{K} \geq \Phi_{t}$}
		\STATE $\eta_{t} \leftarrow \delta_{x_{t}^{\mathrm{A}}} - \delta_{x_{t}^{\mathrm{S}}}$
		\STATE $\alpha_{t} \leftarrow \argmin_{\alpha \in [0, \, \xi_{t}(\{ x_{t}^{\mathrm{A}} \})]} F_{K, \mu}( \xi_{t} - \alpha \, \eta_{t})$
		\STATE $\xi_{t+1} \leftarrow \xi_{t} - \alpha_{t} \, \eta_{t} $ 
		\STATE $\Phi_{t+1} \leftarrow \Phi_{t}$
		\IF {$\alpha_{t} < \xi_{t}(\{ x_{t}^{\mathrm{A}} \})$}
			\STATE $X_{t+1} \leftarrow X_{t}$ 
		\ELSE
			\STATE $X_{t+1} \leftarrow X_{t} \setminus \{ x_{t}^{\mathrm{A}} \}$
		\ENDIF
	\ELSE
		\STATE $x_{t}^{\mathrm{W}} \leftarrow \argmin_{x \in \Omega} \, \left< \nabla F_{K, \mu}(\xi_{t}), \, \delta_{x} \right>_{K}$
		\IF  {$\left< \nabla F_{K, \mu}(\xi_{t}), \, \xi_{t} - \delta_{x_{t}^{\mathrm{W}}} \right>_{K} \geq \Phi_{t}/J$}
			\STATE $\eta_{t} \leftarrow \xi_{t} - \delta_{x_{t}^{\mathrm{W}}}$
			\STATE $\alpha_{t} \leftarrow \argmin_{\alpha \in [0,1]} F_{K, \mu}( \xi_{t} - \alpha \eta_{t} )$
			\STATE $\xi_{t+1} \leftarrow \xi_{t} - \alpha_{t} \eta_{t} $ 
			\STATE $\Phi_{t+1} \leftarrow \Phi_{t}$
			\STATE $X_{t+1} \leftarrow X_{t} \cup \{ x_{t}^{\mathrm{W}} \}$ 
		\ELSE
			\STATE $\xi_{t+1} \leftarrow \xi_{t}$
			\STATE $\Phi_{t+1} \leftarrow \Phi_{t}/2$
			\STATE $X_{t+1} \leftarrow X_{t}$
		\ENDIF
	\ENDIF
\ENDFOR
\end{algorithmic}
\end{algorithm}

Recall that the function $F_{K,\mu}$ is $2$-smooth and convex as indicated in~\eqref{eq:f_2_smooth}. 
Then, 
under the assumption that $K(x,x) = 1$ and $K(x,y) \geq 0$ for any $x,y \in \Omega$,
we obtain the following theorem in a similar manner to Theorems~\ref{thm:only_Lsmooth_Tinv_convergence} and~\ref{thm:lazifiedBPCG}. 

\begin{thm}
Suppose that $K(x,x) = 1$ and $K(x,y) \geq 0$ for any $x,y \in \Omega$. 
Let $\{ \xi_{t} \}_{t=0}^{T} \in \mathscr{M}^{+}(1)$ be the discrete probability measure on $\Omega$
given by Algorithms~\ref{alg:BPCGforKH} or Algorithm~\ref{alg:LazifiedBPCGforKH}. 
Then,  we have
\begin{align}
\gamma_{K}(\xi_{T}, \mu)^{2} = O\left( \frac{1}{T} \right)
\quad
(T \to \infty). 
\notag
\end{align}
\end{thm}

\section{Numerical experiments}
To demonstrate the practical effectiveness of the BPCG algorithm, we compare the performance of BPCG to other state-of-the-art algorithms. To this end we present first numerical experiments in finite-dimensional spaces and compare the performance to (vanilla) Conditional Gradients, Away-Step Conditional Gradients, and Pairwise Conditional Gradients. We then consider the kernel herding setting and compare BPCG to other kernel quadrature methods. 
 
\subsection{Finite dimensional optimization problems}
We report both general primal-dual convergence behavior in iterations and time as well as consider the sparsity of the iterates. 

\subsubsection*{Probaility simplex}
Let $\Delta(n)$ be a probability simplex, i.e.,
\[ \Delta(n)\coloneqq \left\{x\in \mathbb{R}^n \ \middle|\  \sum_{i=1}^n x_i =1, x_i \geq 0 \ (i=1, \ldots, n) \right\}. \]
We consider the following optimization problem:
\begin{align*}
\min_{x\in \mathbb{R}^n}\  & \| x - x_0  \|_2 ^2  \\
\mathrm{s.t.}\ & x \in \Delta (n),
\end{align*}
where $x_0 \in \Delta(n)$.
 \autoref{simplex1} shows the result for the case $n=200$. In terms of convergence in iterations, BPCG has almost the same performance as the Pairwise variant, however in terms of computational time BPCG has the best performance. In this case the LMO is so cheap that using the lazified BPCG algorithm is not advantageous.
\begin{figure}[h]
 \centering
 \includegraphics[width=\linewidth]{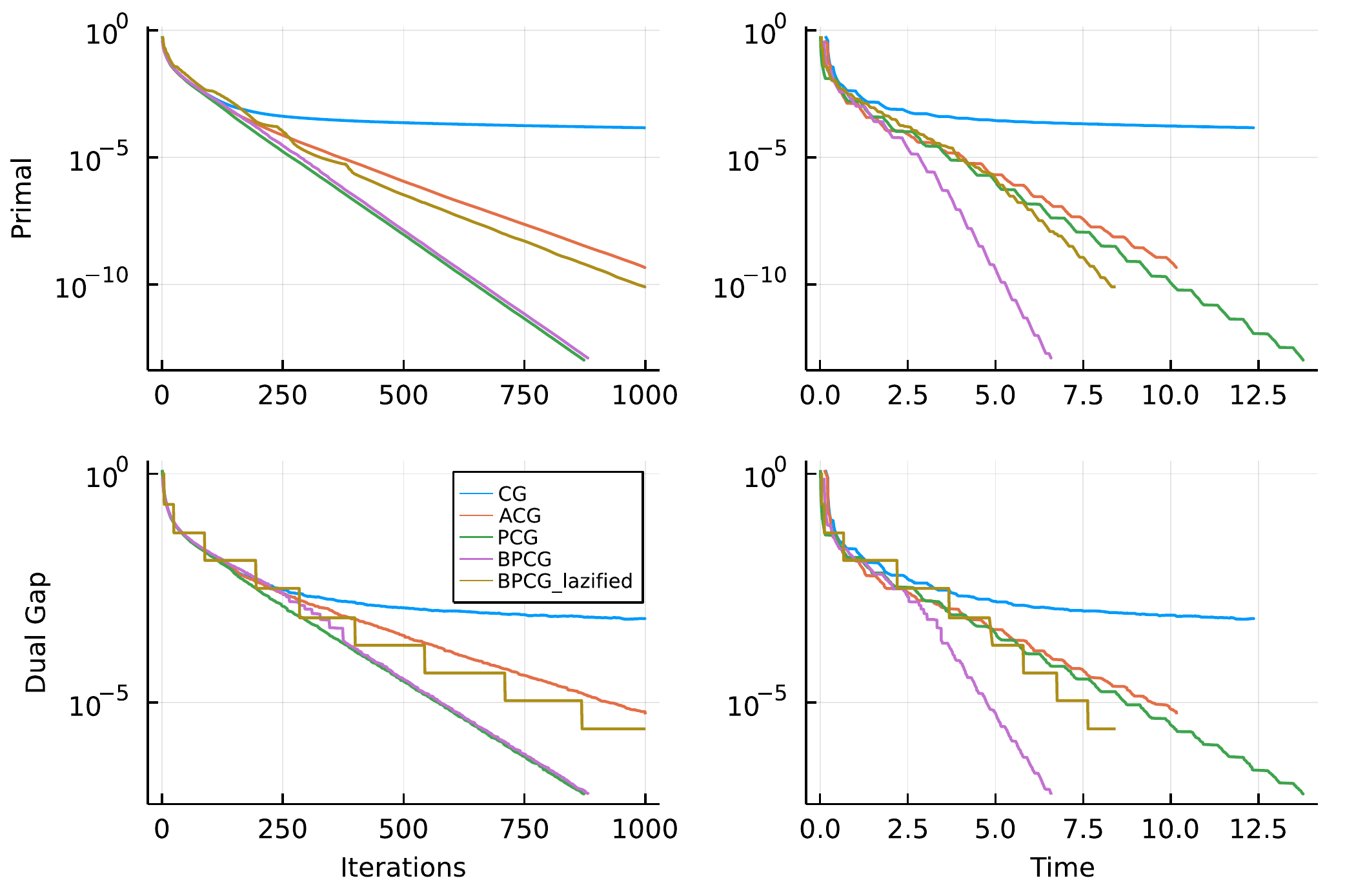}
 \caption{Probability simplex $n=200$}
 \label{simplex1}
 \end{figure}
 In addition, to assess the sparsity of generated solution of each algorithm, we also considered larger setups with $n=500$. The result are shown in  \autoref{simplex2}. It can be observed that both BPCG methods perform slightly (but not much) better than the others with respect to the sparsity, which is expected in this case due to known lower bounds for this type of instance (see \cite{jaggi13fw}).
 
 \begin{figure}[h]
 \centering
  \includegraphics[width=\linewidth]{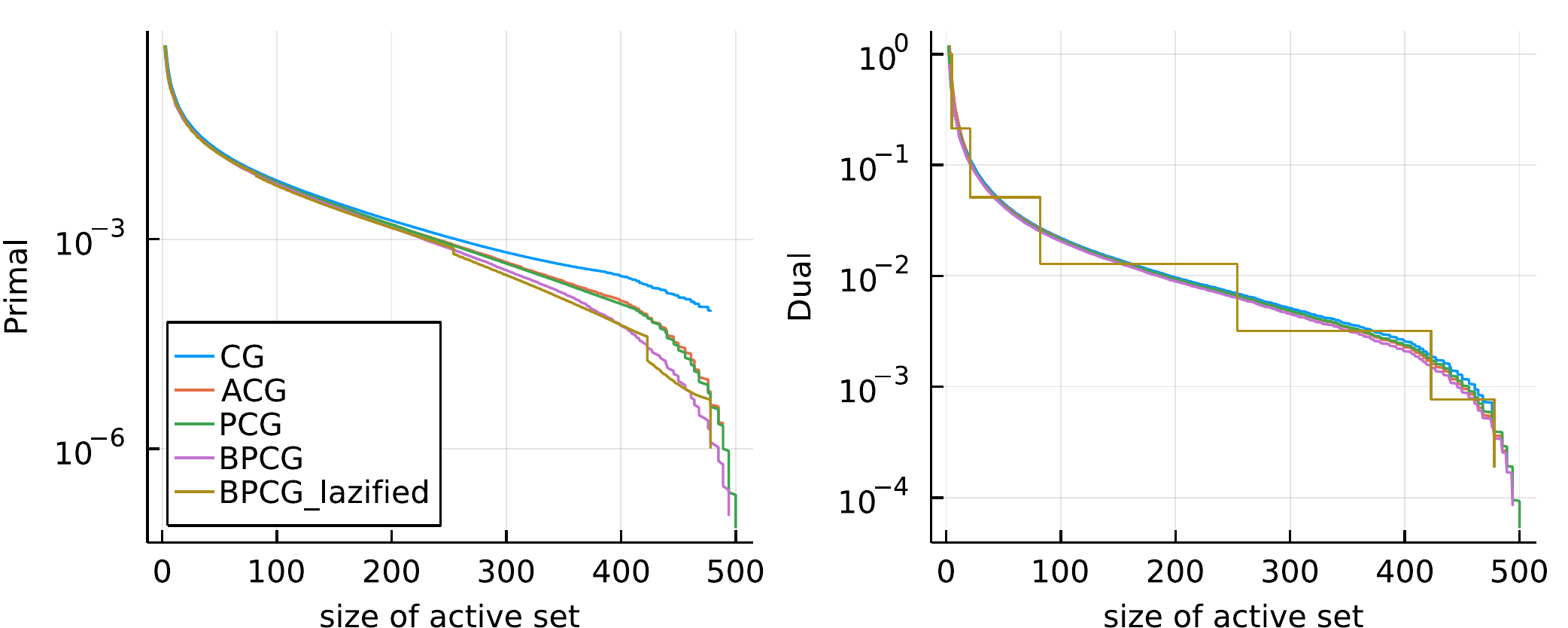}
 \caption{Sparsity: Probability simplex $n=500$}
 \label{simplex2}
\end{figure}

\subsubsection*{Birkhoff polytope}
Let $B(n)$ be the Birkhoff polytope in $\mathbb{R}^{n\times n}$, which is the set of $n\times n$ real values matrices whose entries are all nonnegative and whose rows and columns each add up to $1$.  We consider the following optimization problem 
\begin{align*}
\min_{X\in \mathbb{R}^{n\times n}}\  & \| X - X_0  \|_2 ^2  \\
\mathrm{s.t.}\ & X \in B (n),
\end{align*}
where $X_0 \in \mathbb{R}^{n\times n}$. The result is  \autoref{birkhoff1}:
\begin{figure}[h]
 \centering
 \includegraphics[width=\linewidth]{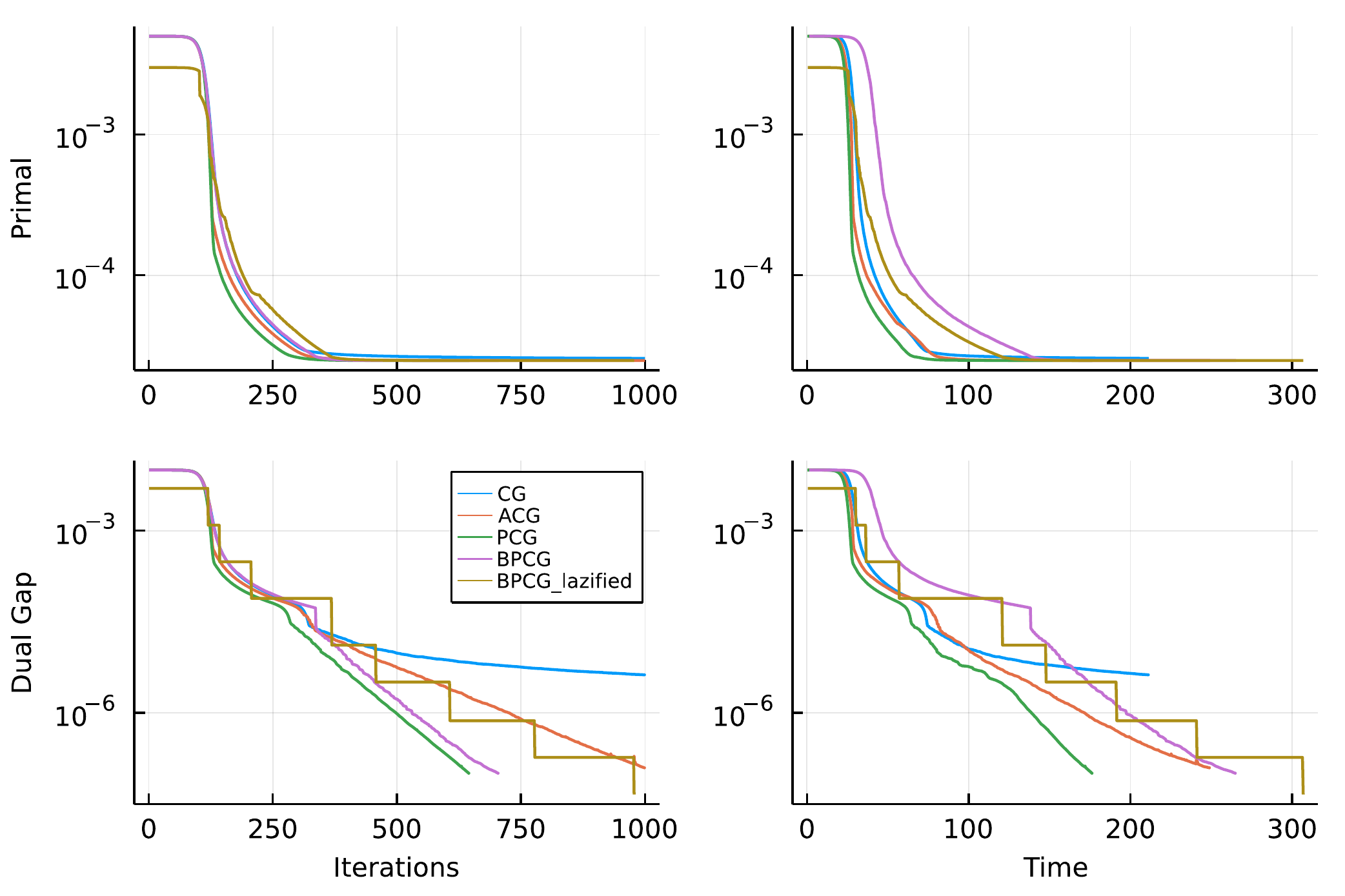}
 \caption{Birkhoff polytope $n=200$}
 \label{birkhoff1}
 \end{figure}
 For the primal value, the five methods have almost same performance. For the dual gap, we see that the Pairwise variant and BPCG are superior to other methods for the number of iterations and for computational time. Here the pairwise variant performed best. 
 
 In addition, we compare solution sparsity in \autoref{birkhoff2}. The two BPCG algorithms perform much better than other methods in terms of the sparsity. 
 
 \begin{figure}[h]
 \centering
  \includegraphics[width=\linewidth]{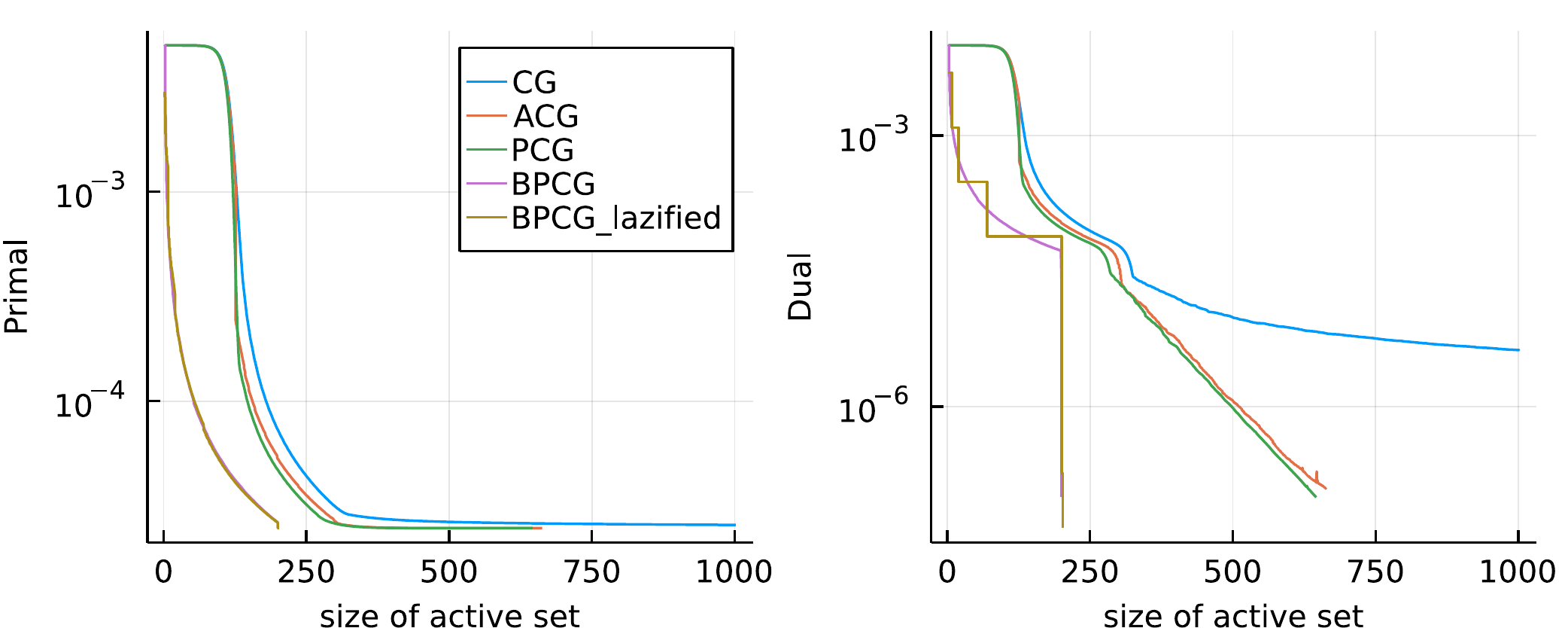}
 \caption{Sparsity: Birkhoff polytope $n=200$}
 \label{birkhoff2}
\end{figure}

\subsubsection*{Matrix completion}
We also considered matrix completion instances over the spectrahedron $S=\{ X \succeq 0 \mid \mathrm{Tr}(X) =1  \} \subset \mathbb{R}^{n\times n}$. The problem is written as  
\[ \min_{X\in S} \sum_{(i, j)\in L} (X_{i, j} - T_{i, j})^2,  \]
where $D= \{ T_{i, j} \mid (i, j) \in L \}$ is an observed data set. We used the data set MovieLens Latest Datasets \url{http://files.grouplens.org/datasets/movielens/ml-latest-small.zip}. The result is given in \autoref{matrix_completion}.
\begin{figure}[h]
 \centering
  \includegraphics[width=\linewidth]{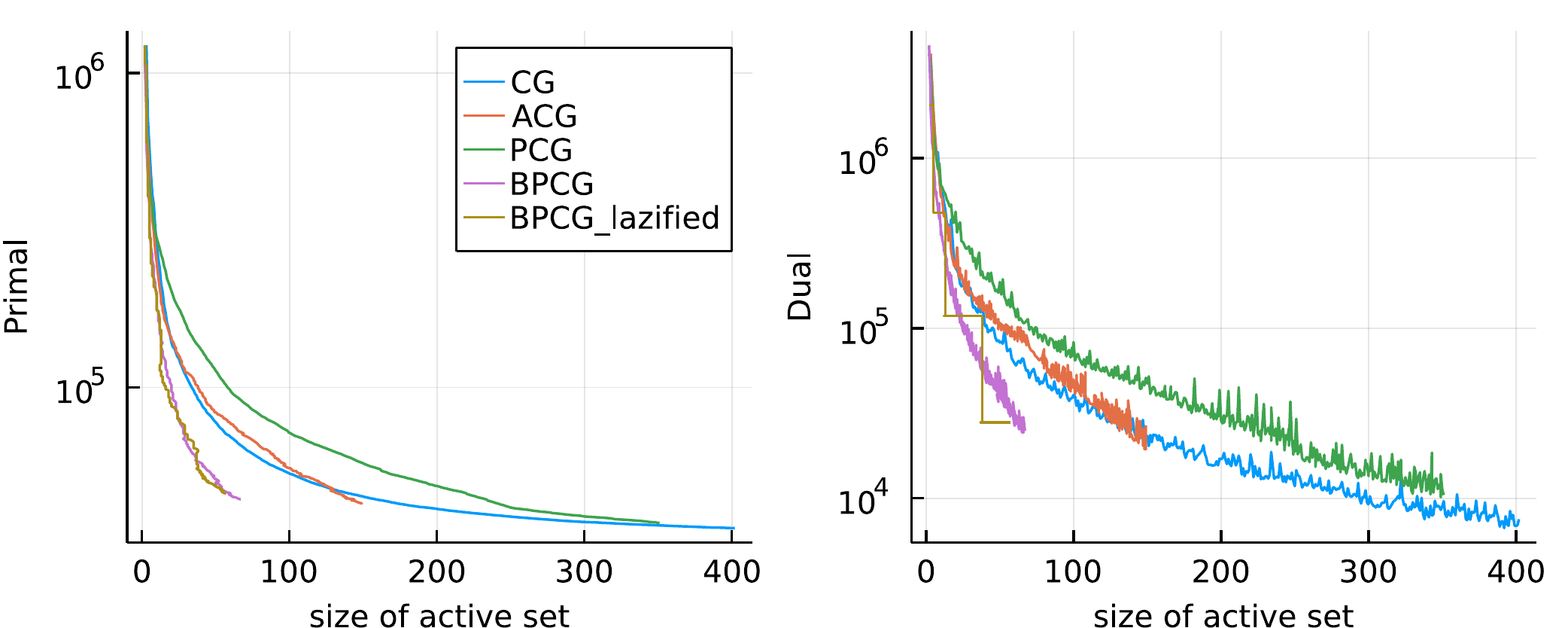}
 \caption{Sparsity: Matrix completion}
 \label{matrix_completion}
\end{figure}
We can observe the effectiveness of BPCG methods in obtaining very sparse solutions.

\subsubsection*{$\ell_p$ norm ball}
We define the $\ell_p$ norm for $p \in \mathbb{R}$ and $x= (x_1, \ldots, x_n) \in \mathbb{R}^n$ as $\| x \|_p = \left( \sum_{i=1}^n | x_i | ^p \right)^{\frac{1}{p}}$ and consider the following optimization problem 
\begin{align*}
\min_{x\in \mathbb{R}^n}\  & \| x - x_0  \|_2 ^2  \\
\mathrm{s.t.}\ & \|x\|_p \leq 1,
\end{align*}
where $\|x_0\|_p \leq 1$. We performed the numerical experiments with $p=5$ and $n=1000$. The result is given in \autoref{l5_ball}. The sparsest solution was obtained via BPCG methods. 
\begin{figure}[h]
 \centering
  \includegraphics[width=\linewidth]{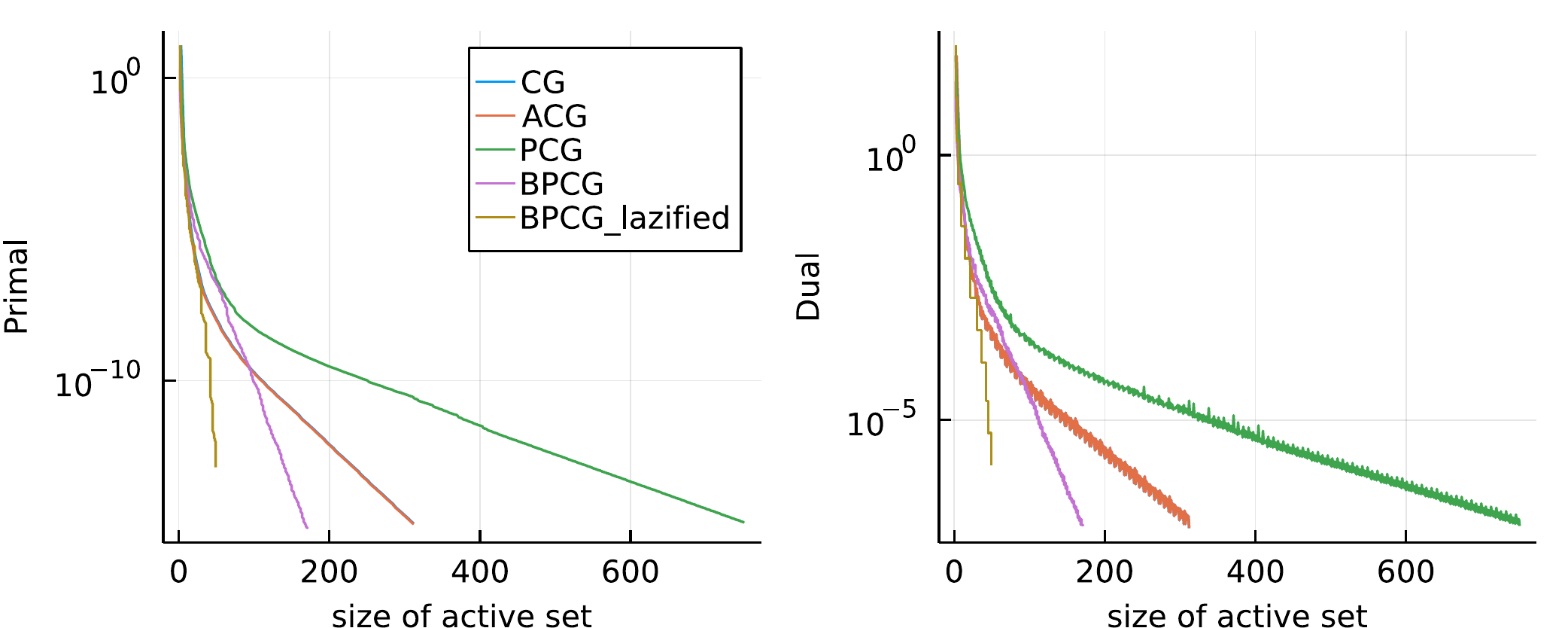}
 \caption{Sparsity: $\ell_5$ norm case}
 \label{l5_ball}
\end{figure}
Note that in this experiment the vanilla Conditional Gradient and Away-step Conditional Gradient behave in exactly the same way and the blue line and orange line in the figure overlap each other. 

\subsection{Kernel herding}
Next, we show the results of the application of BPCG to the kernel herding. We compare the BPCG algorithm and the lazified version with the ordinary kernel herding methods  ``linesearch'' and ``equal-weight'' which correspond to the vanilla Frank Wolfe algorithm whose step size $\alpha_t$ is defined by line search and $\alpha_t = \frac{1}{t+1}$. In addition, we compare BPCG with the Away and Pairwise variants of ``linesearch''; recall that for the latter a theoretical convergence in the infinite-dimensional case is not known. To evaluate the application of BPCG to kernel herding as a quadrature method, we compare it with the popular kernel quadrature method \emph{Sequential Bayesian Quadrature (SBQ)} \citep{10.5555/3020652.3020694} as well as the Monte Carlo method. 

\subsubsection{Mat\'{e}rn kernel case}
We consider the case that the kernel is the Mat\'{e}rn kernel, which has the form
\[ K(x, y)= \frac{2^{1-\nu}}{\Gamma(\nu)} \left( \sqrt{2\nu}\frac{\| x - y \|_2}{\rho} \right)^\nu B_{\nu} \left(\sqrt{2\nu} \frac{ \| x-y \|_2}{\rho} \right), \]
where $B_{\nu}$ is  the modified Bessel function of the second kind and $\rho$ and $\nu$ are positive parameters. 
The Mat\'{e}rn kernel is closely related to Sobolev space and the RKHS $\mathcal{H}_K$ generated by the kernel with smoothness $\nu$ is norm equivalent to the Sobolev space with smoothness $s=\nu + \frac{d}{2}$ (see e.g. \cite{kanagawa2018gaussian,wendland2004scattered}). In addition, the optimal convergence rate of the MMD in Sobolev space is known as $n^{-\frac{s}{d}}$ \citep{novak2006deterministic}. In this section, we use the parameter $(\rho, \nu)= (\sqrt{3}, \frac{3}{2}), (\sqrt{5}, \frac{5}{2})$ since the kernel has explicit forms with these parameters. 

 \begin{figure}[h]
    \begin{tabular}{cc}
      %---- 最初の図 ---------------------------
      \begin{minipage}[t]{0.5\linewidth}
        \centering
        \includegraphics[keepaspectratio, scale=0.35]{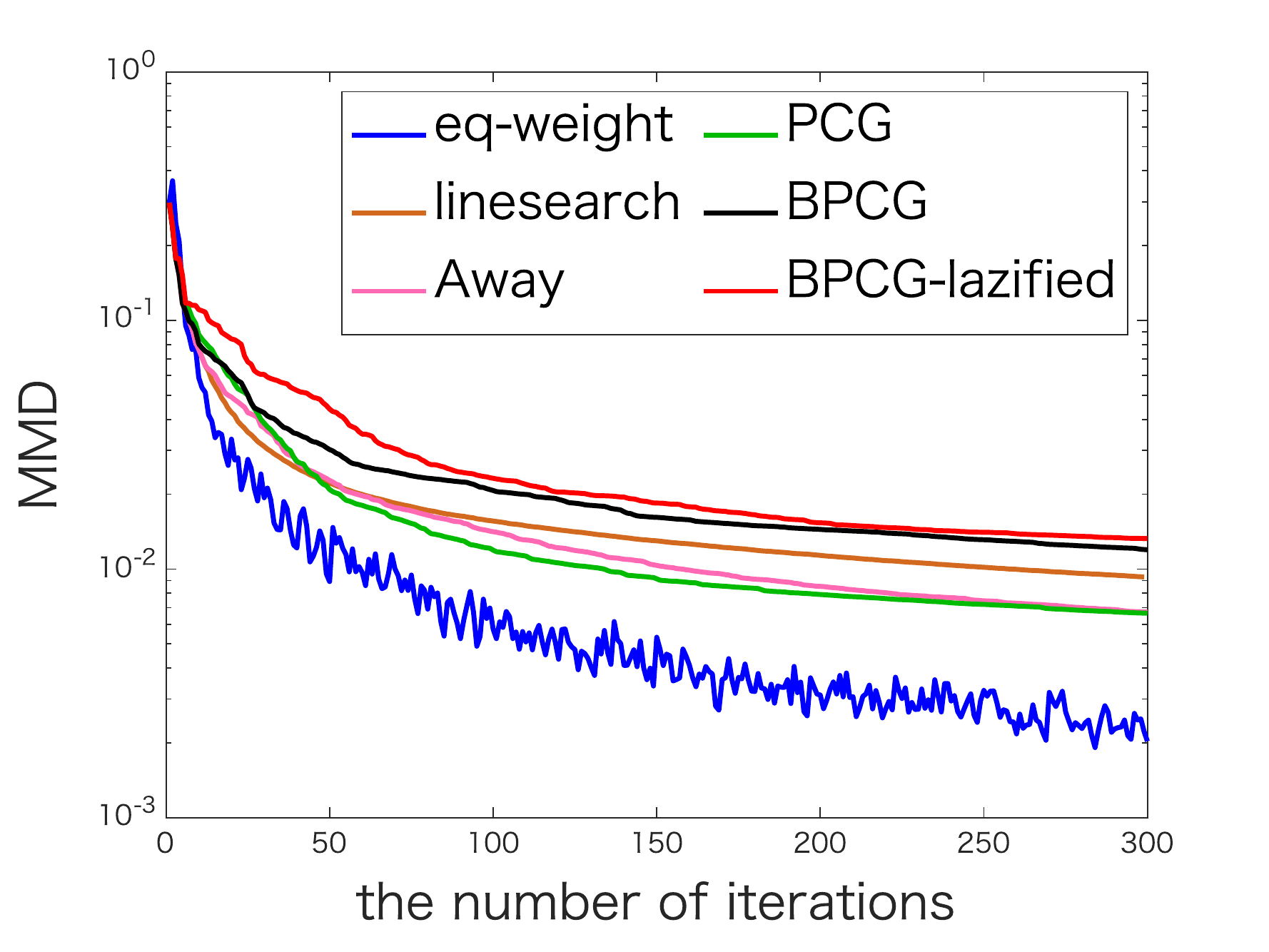}
        \subcaption{MMD for the number of iterations}
      \end{minipage} 
      %---- 2番目の図 --------------------------
      \begin{minipage}[t]{0.5\linewidth}
        \centering
        \includegraphics[keepaspectratio, scale=0.35]{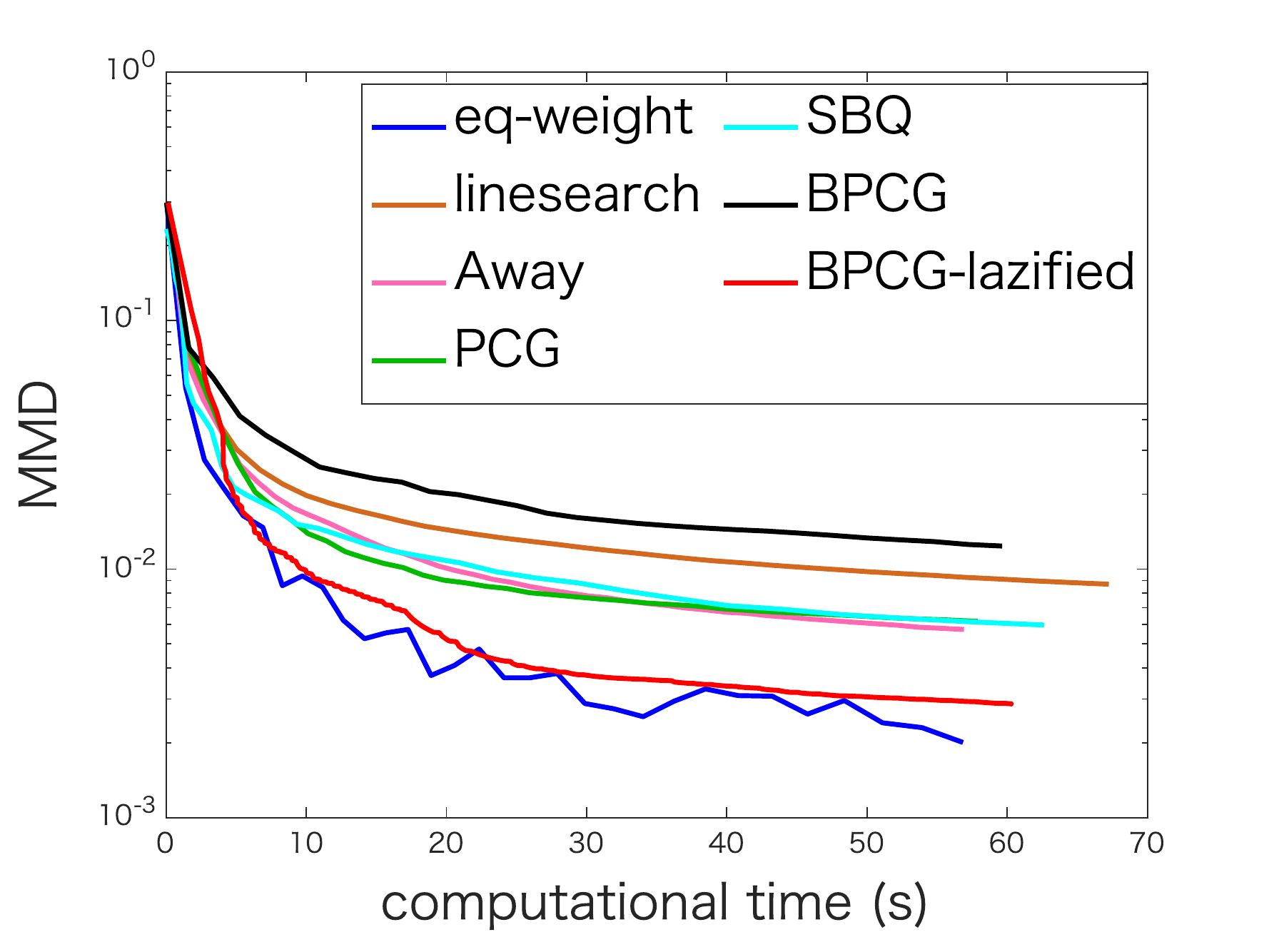}
        \subcaption{MMD for computational time}
      \end{minipage}\\
       \begin{minipage}[t]{0.5\linewidth}
        \centering
        \includegraphics[keepaspectratio, scale=0.35]{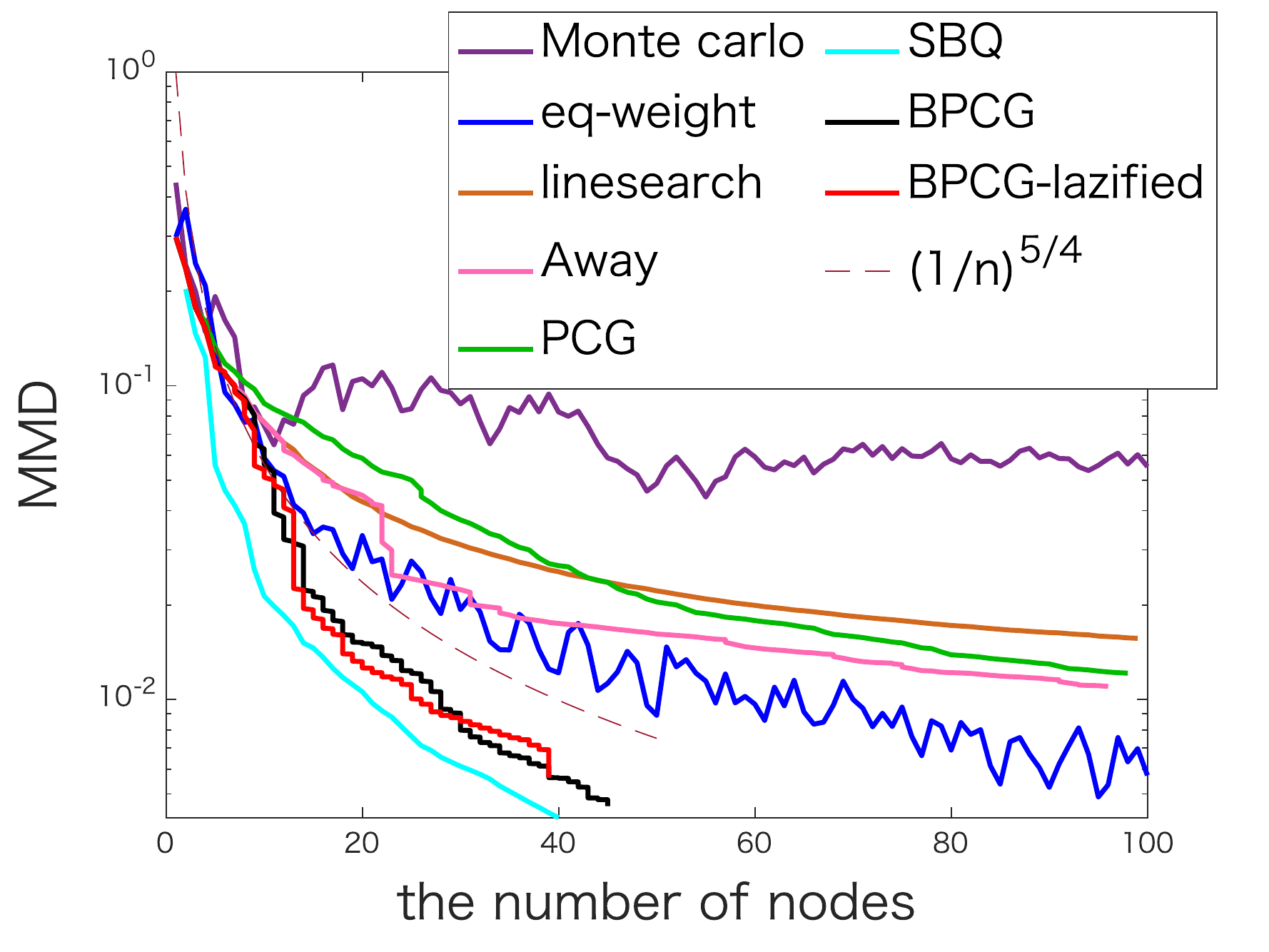}
        \subcaption{MMD for the number of nodes}
      \end{minipage}
      %---- 図はここまで ----------------------
    \end{tabular}
    \caption{Mat\'{e}rn kernel ($\nu=3/2$)}
    \label{matern32}
  \end{figure}

The domain $\Omega$ is $[-1, 1]^d$ and the probability distribution is a uniform distribution. 
First, we see the case $\nu=\frac{3}{2}$ and $d=2$. In this case, the optimal rate is $n^{-\frac{5}{4}}$. The result is  \autoref{matern32}. For sparsity, BPCG methods have significant performance and they achieve the convergence speed comparable to the optimal rate. For computational time, the lazified BPCG have good performance. This is because in the lazified algorithm, we only access the active set and this reduces the computational cost considerably. Although, BPCG methods are not superior to other methods for the number of the iterations, we can see the effectiveness of BPCG methods.

 Moreover, the same good performance can be seen in the case $\nu= \frac{5}{2}$ and $d=2$, which is shown in  \autoref{matern52}. BPCG algorithms also achieves the convergence rate which is competitive with the optimal rate $n^{-\frac{7}{4}}$.

 \begin{figure}[t]
 \begin{tabular}{cc}
 \begin{minipage}[h]{0.45\linewidth}
    \centering
 \includegraphics[keepaspectratio, scale=0.3]{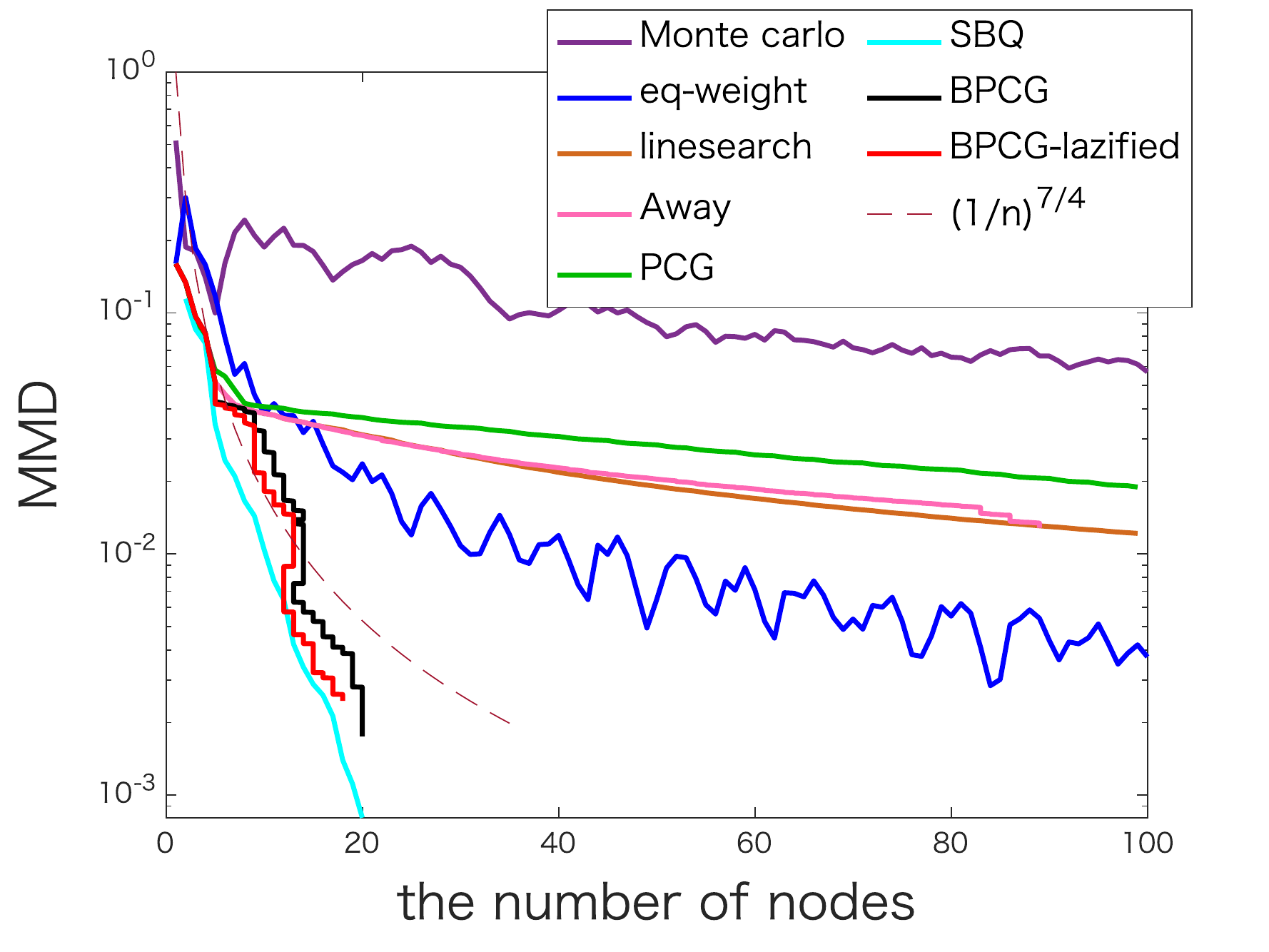}
% \caption{Mat\'{e}rn kernel ($\nu=5/2$)}
 \end{minipage}
 \begin{minipage}[h]{0.45\linewidth}
 \centering
 \includegraphics[keepaspectratio, scale=0.3]{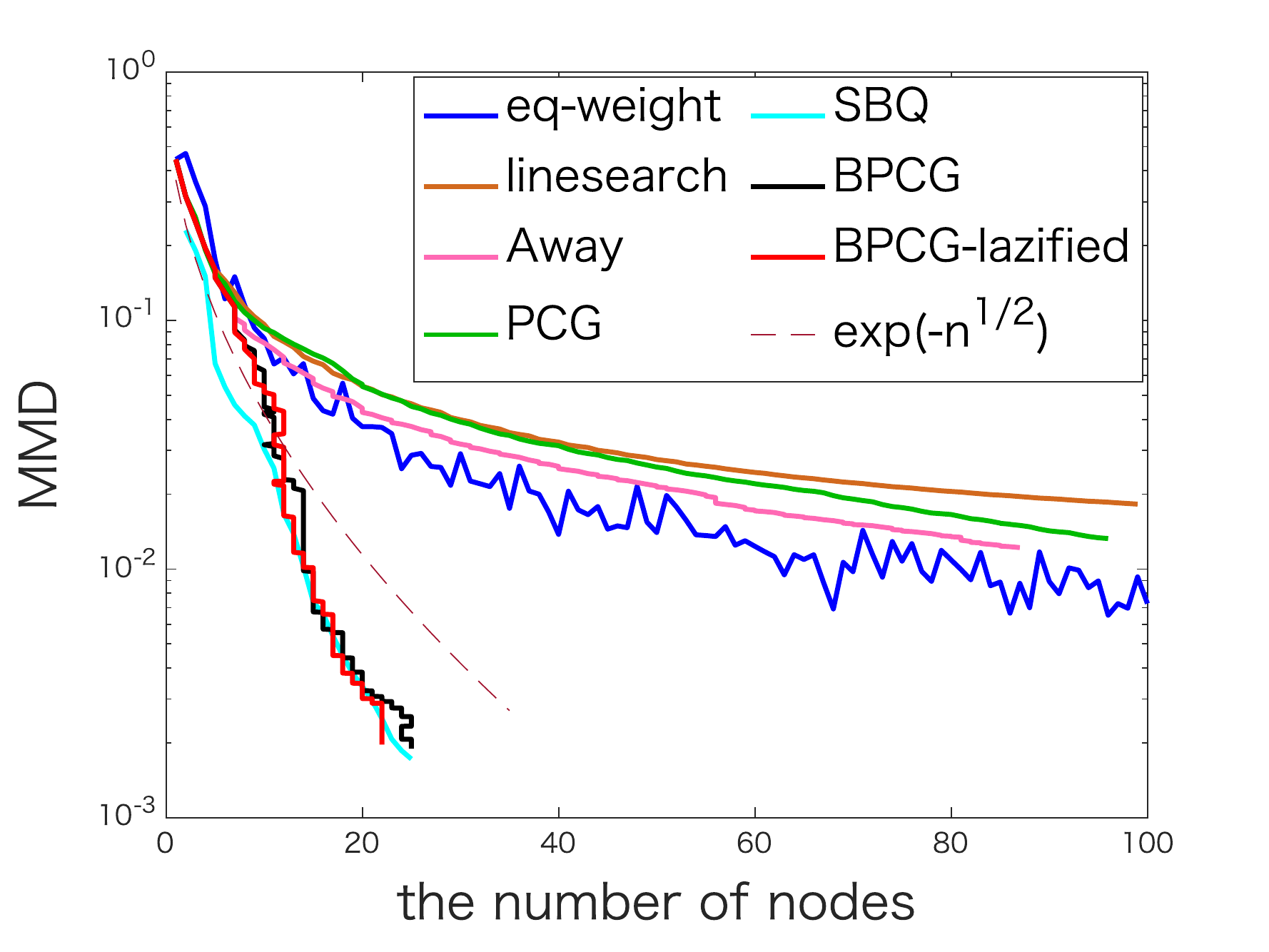}
% \caption{Gaussian kernel}
  \end{minipage}
  \end{tabular}
 \caption{Mat\'{e}rn kernel ($\nu=5/2$) (left) and Gaussian kernel (right)}
 \label{gaussian}
 \label{matern52}
\end{figure}

\subsubsection{Gaussian kernel}
Next, we treat the gaussian kernel $K(x, y)= \mathrm{exp}(-\| x - y \|_2 ^2)$. The domain is $\Omega= [-1, 1]^d$ and the density function of the  distribution on $\Omega$ is $\frac{1}{C} \mathrm{exp}\left( -\| x \|^2  \right)$, where $C= \int_{\Omega} \mathrm{exp}\left( -\| x \|^2  \right) \mathrm{d}x$. In this case, we see the case $d=2$ and the result is \autoref{gaussian}. We note that $\mathrm{exp}(-n^{\frac{1}{2}})$ is the exponential factor of the upper bound of convergence rate which appears in several works on gaussian kernel, for example, \cite{wendland2004scattered,karvonen2021integration}. We see that BPCG methods are competitive with SBQ and the $\mathrm{exp}(-n^{\frac{1}{2}})$ rate and outperform significantly other methods.

In addition, we performed the experiments for a mixture Gaussians on $[-1,1]^2$. As shown in \autoref{mixture_gauss}, we have almost the same result as that of the ordinary Gaussian distribution. We note that in the left figure in \autoref{mixture_gauss}, the distribution function takes large values as the color gets closer to yellow and takes small values as the color gets closer to blue. 

\begin{figure}[h]
 \begin{tabular}{cc}
 \begin{minipage}[h]{0.45\linewidth}
    \centering
 \includegraphics[keepaspectratio, scale=0.3]{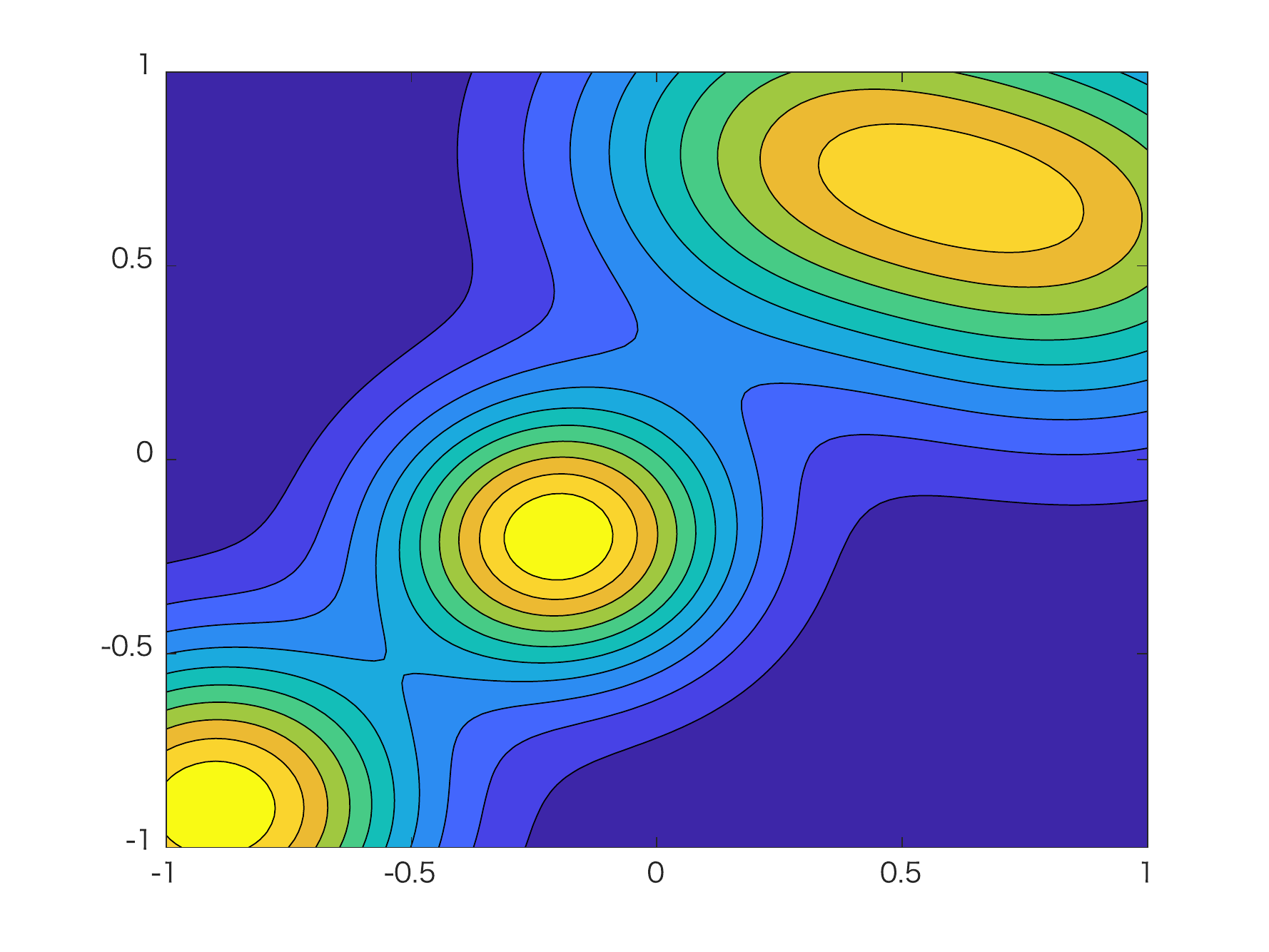}
% \caption{Mat\'{e}rn kernel ($\nu=5/2$)}
 \end{minipage}
 \begin{minipage}[h]{0.45\linewidth}
 \centering
 \includegraphics[keepaspectratio, scale=0.3]{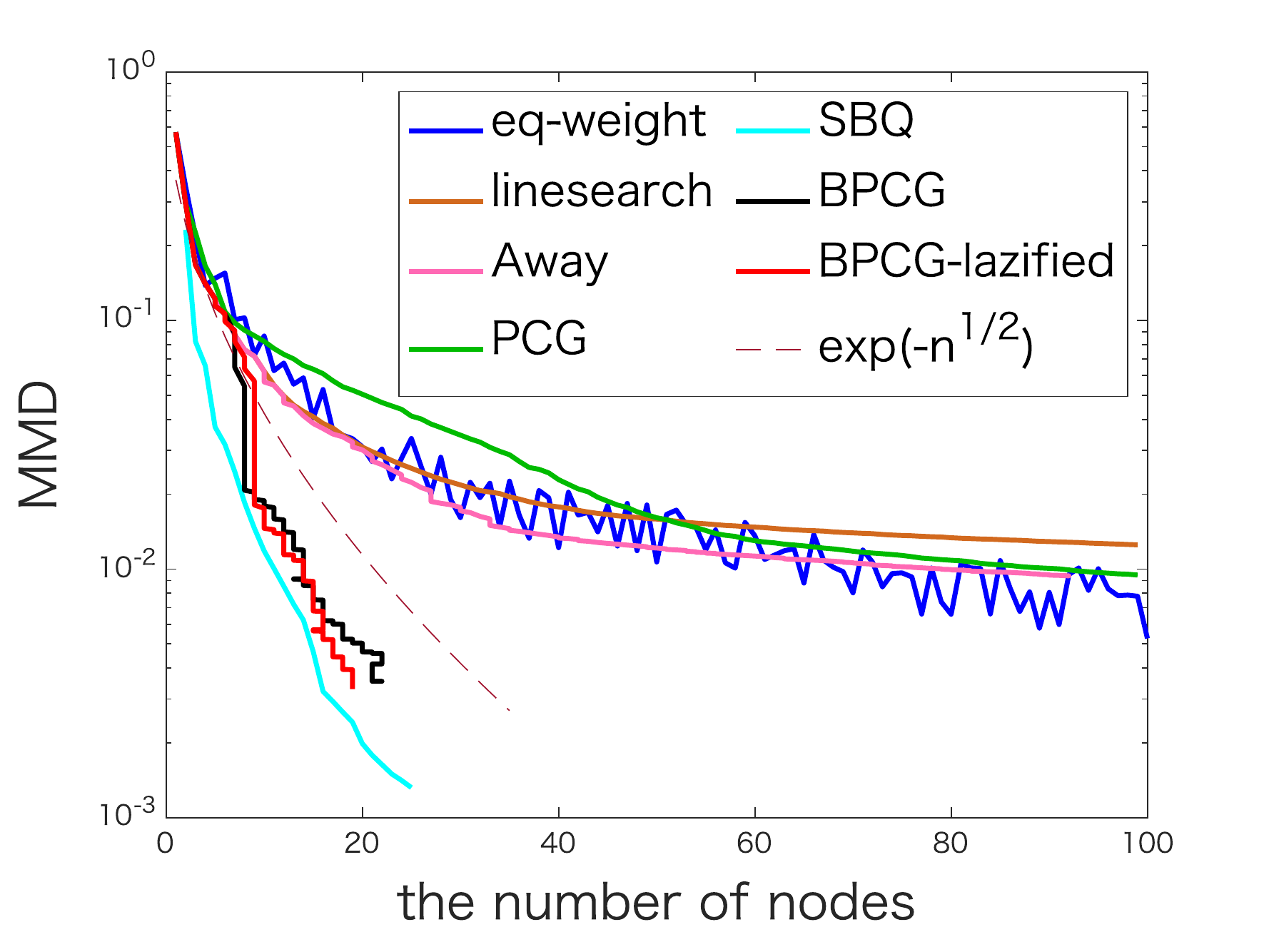}
% \caption{Gaussian kernel}
  \end{minipage}
  \end{tabular}
 \caption{Contour of the mixture gausssian distribution (left) and MMD for the number of nodes(right). }
\label{mixture_gauss}
\end{figure}

% conclusion
\section{Conclusion}

%<<<<<<< HEAD
%The proposed algorithm BPCG yields sparse solutions fast. 
%It does not exhibit any swap steps, which underpins that good property.
%We have analyzed its convergence property and
%observed its real performance via the numerical experiments. 
%The BPCG works well for application to kernel herding, the infinite-dimensional case, 
%in that it provides small MMD with small number of nodes. 
%A main topic for future work will be tighter estimate for the convergence rate of the BPCG for kernel herding.
%=======
The proposed Blended Pairwise Conditional Gradient (BPCG) algorithms yields very sparse solutions fast with very high speed both in convergence in iterations as well as time. 
It does not exhibit any swap steps, which provides state-of-the-art convergence guarantees for the strongly convex case as well as applies application to the infinite-dimensional case. We have analyzed its convergence property and
exemplified its real performance via the numerical experiments. 
The BPCG works well for application to kernel herding, the infinite-dimensional case, in that it provides small MMD with a small number of nodes. 
A main avenue for future work will be tighter estimates for the convergence rate of the BPCG in various cases, in particular with regards to sparsity.

\bibliography{refs.bib}
\bibliographystyle{apalike}

\end{document}

%% file: preamble.tex
\usepackage[T1]{fontenc}
\usepackage[utf8]{inputenc}

\usepackage{lmodern}
\usepackage{microtype}
\usepackage{csquotes}

\usepackage{amsmath}
\usepackage{amssymb}
\usepackage{amsthm}
\usepackage{mathtools}
\usepackage{color}
\usepackage{graphicx}
\usepackage{wrapfig}
\usepackage{float}
\usepackage{url}
\usepackage{boites}
\usepackage{mathrsfs}
\usepackage{algorithm}
\usepackage{algorithmic}
\usepackage{tabularx}
\usepackage{multicol}
\usepackage[subrefformat=parens]{subcaption}
\usepackage{hyperref}

\theoremstyle{definition}

\newtheorem{thm}{Theorem}
\newtheorem{lem}[thm]{Lemma}
\newtheorem{rem}{Remark}

\newtheorem*{cor*}{Corollary}

\numberwithin{dfn}{section}
\numberwithin{thm}{section}
\numberwithin{rem}{section}
\numberwithin{equation}{section}
\numberwithin{ex}{section}
\numberwithin{prob}{section}
\numberwithin{ans}{section}
\numberwithin{assum}{section}
\numberwithin{prop}{section}

\newcommand{\zissu}{\mathbb{R}} % real number
\newcommand{\argmax}{\mathop{\mathrm{argmax}}}
\newcommand{\argmin}{\mathop{\mathrm{argmin}}}

\usepackage{url}
\usepackage{breakurl}
\usepackage{comment}